\documentclass[12pt,reqno]{amsart}

\usepackage{amsmath}
\usepackage{amsfonts}
\usepackage{amssymb}
\usepackage{amsthm}
\usepackage{esint}
\usepackage{mathrsfs}
\usepackage{mathtools}
\usepackage{tensor}

\usepackage{fullpage}
\usepackage{ucs}
\usepackage[utf8x]{inputenc}
\usepackage{color}
\usepackage{setspace}
\usepackage{verbatim}
\usepackage{url}
\usepackage{hyperref}
\usepackage{enumerate}
\usepackage{enumitem}
\usepackage{bm}

\newcommand{\RR}{\mathbf{R}}

\newcommand{\NN}{\mathbf{N}}

\newcommand{\cA}{\mathcal{A}}

\newcommand{\cH}{\mathcal{H}}
\newcommand{\cI}{\mathcal{I}}

\newcommand{\cL}{\mathcal{L}}

\newcommand{\cO}{\mathcal{O}}

\newcommand{\cQ}{\mathcal{Q}}
\newcommand{\cR}{\mathcal{R}}

\newcommand{\cU}{\mathcal{U}}

\DeclareMathOperator{\lspan}{span}

\DeclareMathOperator{\tr}{tr}

\DeclareMathOperator{\ind}{ind}
\DeclareMathOperator{\nul}{nul}

\DeclareMathOperator{\support}{spt}

\DeclareMathOperator{\divg}{div}

\DeclareMathOperator{\sing}{sing}
\DeclareMathOperator{\reg}{reg}

\DeclareMathOperator{\riem}{Rm}
\DeclareMathOperator{\ricc}{Ric}

\DeclareMathOperator{\sff}{\mathrm{I\!I}}

\newcommand{\indt}[1]{{\bf 1}_{#1}}

\newcommand{\restr}{\mathbin{\vrule height 1.6ex depth 0pt width
0.13ex\vrule height 0.13ex depth 0pt width 1.3ex}}
\newcommand{\eps}{\varepsilon}

\newcommand{\weaklyto}{\rightharpoonup}

\newcommand{\energyunit}{\mathfrak{e}_0}

\newcommand{\mfh}{\mathfrak{h}}
\newcommand{\mbh}{\mathbb{H}}
\newcommand{\mbi}{\mathbb{I}}
\newcommand{\mbj}{\mathbb{J}}
\newcommand{\mbk}{\mathbb{K}}
\newcommand{\mbl}{\mathbb{L}}

\theoremstyle{definition} \newtheorem{defi}{Definition}
\theoremstyle{definition} \newtheorem{rema}[defi]{Remark}
\theoremstyle{definition} \newtheorem{theo}[defi]{Theorem}
\theoremstyle{definition} 
\theoremstyle{definition} \newtheorem{coro}[defi]{Corollary}
\theoremstyle{definition} \newtheorem{lemm}[defi]{Lemma}
\theoremstyle{definition} 
\theoremstyle{definition} 
\theoremstyle{definition} 
\theoremstyle{definition} 
\theoremstyle{definition} 

\allowdisplaybreaks

\title{Variational aspects of phase transitions with prescribed mean curvature}
\author{Christos Mantoulidis}
\address{Rice University, Dept. of Mathematics, 6000 Main St, MS 136, Houston, TX 77005}
\email{christos.mantoulidis@rice.edu}
\date{}

\begin{document}

\begin{abstract}
	We study the spectrum of phase transitions with prescribed mean curvature in Riemannian manifolds. These phase transitions are solutions to an inhomogeneous semilinear elliptic PDE that give rise to diffuse objects (varifolds) that limit to hypersurfaces, possibly with singularities, whose mean curvature is determined by the ``prescribed mean curvature'' function and the limiting multiplicity. We establish upper bounds for the eigenvalues of the diffuse problem, as well as the more subtle lower bounds when the diffuse problem converges with multiplicity one. For the latter, we also establish asymptotics that are sharp to order $o(\eps^2)$ and $C^{2,\alpha}$ estimates on multiplicity-one phase transition layers.
\end{abstract}

\maketitle

\section{Introduction} \label{sec:intro}

Let $(M^n, g)$ be a Riemannian manifold. Consider the semilinear elliptic PDE
\begin{equation} \label{eq:ac.pde.h}
	\eps^2 \Delta u = W'(u) + \eps \mfh	
\end{equation}
for $\eps > 0$, smooth $u$, $\mfh : M \to \RR$, and a smooth double-well potential $W : \RR \to \RR$ satisfying:
\begin{itemize}
	\item $W(x) \geq 0$ and vanishes if and only if $x = \pm 1$,
	\item $W'(0) = 0$, $W''(0) \neq 0$, and $xW'(x) < 0$ for $x \in (0, 1)$,
	\item $W''(x) \geq \kappa > 0$ for $x \in \RR \setminus (-1+\beta,1-\beta)$ for some $\beta \in (0, 1)$, and
	\item $W(x) = W(-x)$ for all $x$;
\end{itemize}
a canonical choice is $W(x) = \tfrac14 (1-x^2)^2$. This PDE describes the Gibbs--Thomson law, and it also relates to the Van der Waals--Cahn--Hilliard theory of phase transitions (\cite{HutchinsonTonegawa00, RogerTonegawa08}). Solutions $u$ of \eqref{eq:ac.pde.h} are critical points (see Section \ref{sec:variations}) of the energy
\begin{equation} \label{eq:ac.energy.h}
	E_{\eps,\mfh}[u] := \int_M (\tfrac{\eps}{2} |\nabla u|^2 + \eps^{-1} W(u) + \mfh u) \, d\mu_g. 
\end{equation}
When $M$ is not a closed (i.e., compact, no boundary) manifold, one simply works locally in the interior of $M$. Since we are interested in variational applications to geometric problems, we will assume that $M$ is closed, except when otherwise stated. A solution $(u, \eps, \mfh)$ of \eqref{eq:ac.pde.h} with finite $E_{\eps,\mfh}$ energy gives rise to a codimension-1 diffuse varifold on $(M, g)$, whose induced Radon measure is $\eps |\nabla u|^2 \, d\mu_g$ (\cite{HutchinsonTonegawa00}). Under certain uniform bounds on our data $(u, \eps, \mfh)$, these diffuse varifolds subsequentially converge in $M$, as $\eps \to 0$, to a codimension-1 integral varifold $V$, which we call a \emph{limiting varifold} that bounds, in a certain sense, a Caccioppoli set $\Omega$, which we call a \emph{limiting enclosed domain} (\cite{HutchinsonTonegawa00, RogerTonegawa08, Guaraco}). In fact, $u \to \indt{M \setminus \bar \Omega} - \indt{\Omega}$ in the $L^1$ sense, $\{ u = 0 \} \to \support \Vert V \Vert$ in the Hausdorff sense, and $\support \Vert V \Vert$ consists of two portions:
\begin{enumerate}[label=(\Alph*)]
	\item \label{rema:intro.v.omega.odd} the portion where the density $\Theta^{n-1}(V, \cdot)$ is odd, which a.e. coincides with $\partial^* \Omega$, and
	\item \label{rema:intro.v.omega.even} the portion where the density $\Theta^{n-1}(V, \cdot)$ is even, which a.e. misses $\partial^* \Omega$.
\end{enumerate}
Here, $\partial^* \Omega$ denotes the reduced boundary of the Caccioppoli set $\Omega$. The weak mean curvature vector $\bm{H}$ of $V$ depends on the density $k = \Theta^{n-1}(V, \cdot)$. In fact,
\begin{enumerate}[label=(\alph*)]
	\item \label{rema:intro.h.odd} when $k$ is odd: $\bm{H} = -2 (k \energyunit)^{-1} \mfh \bm{n}$ a.e., with $\bm{n}$ being the unit vector pointing into the $+1$ region, $\energyunit$ being the squared $L^2$ energy of the heteroclinic solution; while,
	\item \label{rema:intro.h.even} when $k$ is even: $\bm{H} = \bm{0}$ a.e.
\end{enumerate}
As a result,  $(V; \Omega)$ is a critical point (in the sense of ambient deformations---see Section \ref{sec:variations}) of the \emph{prescribed mean curvature functional}\footnote{If $V$ is a smooth multiplicity-one hypersurface, then $A_\mfh[V; \Omega]$ measures the $(n-1)$-dimensional area of $V$ minus the bulk integral of $\mfh$ in the region $\Omega$ enclosed by $V$. Smooth multiplicity-one critical points $(V; \Omega)$ of this functional will have mean curvature equal to $2 \energyunit^{-1} \mfh$.}
\begin{equation} \label{eq:area.h.functional}
	A_{2 \energyunit^{-1} \mfh}[V; \Omega] := \Vert V \Vert(M) - \int_\Omega 2 \energyunit^{-1} \mfh.
\end{equation}
It has been shown that $k \neq 3, 5, \ldots$ unless $\mfh = 0$ (\cite{RogerTonegawa08}). Thus,
\begin{enumerate}[label=(\alph*')]
	\item \label{rema:intro.h.odd.sharp} when $k$ is odd and $\mfh \neq 0$: $k=1$ and $\bm{H} = -2 \energyunit^{-1} \mfh \bm{n}$ a.e.; while,
	\item \label{rema:intro.h.even.sharp} when $k$ is even or $\mfh = 0$: $\bm{H} = \bm{0}$ a.e.
\end{enumerate}
We focus on solutions $(u, \eps, \mfh)$ that are produced by variational methods (usually min-max). We are interested in understanding their Morse index and nullity. We briefly recall some relevant results for $\mfh \equiv 0$:
\begin{itemize}
	\item For $n \geq 3$, we know from \cite{TonegawaWickramasekera12, Guaraco, Gaspar, GasparGuaraco, Hiesmayr} that solutions $(u, \eps)$ of \eqref{eq:ac.pde.h} with uniformly bounded $E_{\eps,0}[u]$, and uniformly bounded Morse index must subsequentially collapse as $\eps \to 0$, possibly with multiplicity, to a limiting varifold $V$ that is smooth outside a set of ambient codimension $8$ and has Morse index (see Section \ref{sec:variations})
		\begin{equation} \label{eq:index.upper.old}
			\ind(V) \leq \lim_i \ind_{E_{\eps_i,0}}(u_i).
		\end{equation}
		In the reverse direction, we know from \cite{ChodoshMantoulidis:multiplicity-one} that
		\begin{equation} \label{eq:index.lower.old}
			\ind(V) + \nul(V) \geq \lim_i (\ind_{E_{\eps_i,0}}(u_i) + \nul_{E_{\eps_i,0}}(u_i))
		\end{equation}
		when $V$ is smooth with multiplicity one (in any dimension, including $n=2$).
	\item For $n=2$, the situation is subtle. First, the singular set has at most $\leq \lim_\eps \ind_{E_{\eps,0}}(u)$ points (\cite{Tonegawa05}). A finer study of the pre-limit behavior of $(u, \eps)$ shows that solutions to \eqref{eq:ac.pde.h} with bounded $E_{\eps,0}[u]$ and Morse index $\leq 1$ must collapse as $\eps \to 0$ to either smoothly embedded geodesics (possibly with multiplicity) or to  smoothly immersed geodesics with multiplicity one and a single non-embedded point that resembles an ``X;'' more generally, the total density of singular points is bounded from above by a function of $\lim_\eps \ind_{E_{\eps,0}}(u)$ (\cite{Mantoulidis}).
\end{itemize}
In this paper:
\begin{itemize}
	\item We generalize \eqref{eq:index.upper.old} and \eqref{eq:index.lower.old} to all $\mfh$. 
	\item We also establish regularity theory needed to extend the $n=2$ bullet point above to nonzero $\mfh$ in future work. 
\end{itemize}
We do the above via a sharp to order $o(\eps^2)$ understanding of $u$.

Fix a background closed Riemannian manifold $(M^n, g)$. Consider a sequence $\{ (u_i, \eps_i, \mfh_i) \}_{i=1,2,\ldots}$ such that, for a fixed $\alpha \in (0, 1)$,
\begin{equation} \label{eq:structure.assumptions}
	\eps_i > 0, \; u_i \text{ is a critical point of } E_{\eps_i, \mfh_i}, \; \lim_i \eps_i = 0, \; \sup_i E_{\eps_i,\mfh}(u_i) + \Vert \mfh_i \Vert_{C^{3,\alpha}(M)} < \infty.
\end{equation}
Recall that, after passing to a subsequence, the diffuse $(n-1)$-varifolds associated with $(u_i, \eps_i)$ converge (as $i \to \infty$) to a limiting integral $(n-1)$-varifold, $V$, and a limiting enclosed domain, $\Omega$. After possibly passing to a further subsequence, $\mfh_i$ converge in $C^{3,\alpha'}$ with $\alpha' \in (0,1)$ to $\mfh \in C^{3,\alpha}(M)$. In what follows, always assume that we pass to subsequences as necessary:

\begin{theo} \label{theo:higher.dim}
	Suppose $n \geq 2$, $\alpha \in (0, 1)$. Let $\{(u_i, \eps_i, \mfh_i)\}_{i=1,2,\ldots}$ be as in \eqref{eq:structure.assumptions}. Let $V$, $\Omega$ be as before, and $\mfh = \lim_i \mfh_i$.
	\begin{enumerate}[label=(\alph*)]
		\item Suppose that $\sup_i \ind_{E_{\eps_i,\mfh_i}}(u_i) < \infty$. We may estimate the Morse index of $(V; \Omega)$ by
			\begin{equation} \label{eq:higher.dim.index.upper.h}
				\ind_{A_{2\energyunit^{-1} \mfh}}(V; \Omega) \leq \lim_i \ind_{E_{\eps_i,\mfh_i}}(u_i).
			\end{equation}
			See Section \ref{sec:variations} for notation and Theorem \ref{theo:index.upper} for a stronger result with weaker (i.e., $W^{2,p}$ rather than $C^{3,\alpha}$) assumptions on $\mfh_i$.
		\item Suppose $U \subset M$ is an open set such that $\bar U \cap \sing V = \emptyset$, $\Gamma' := \support \Vert V \Vert \cap U$ is an embedded $(n-1)$-manifold, $\bar \Gamma' \setminus \Gamma'$ is smooth, and $\Theta^{n-1}(V, \cdot) \equiv 1$ along $\Gamma'$. Then, for all $\tau \in (-1, 1)$, there exists an open set $U' \subset M$ containing $\Gamma'$ such that $\Gamma' \cap U'$ is the $C^{2,\alpha}$ limit of $\Gamma_{i,\tau}' := \{ u_i = \tau \} \cap U'$ as $i \to \infty$. The mean curvature $H_{i,\tau}$ of $\Gamma_{i,\tau}'$ satisfies
			\begin{equation} \label{eq:higher.dim.mean.curv}
				\Vert H_{i,\tau} - 2 \energyunit^{-1} \mfh_i \Vert_{C^0(\Gamma_{i,\tau}')} + \eps^\alpha [H_{i,\tau} - 2 \energyunit^{-1} \mfh_i ]_{\alpha,\Gamma_{i,\tau}'} = O(\eps).
			\end{equation}
		\item If $\Theta^{n-1}(V, \cdot) \equiv 1$ everywhere on $\support \Vert V \Vert$, then we may estimate the Morse index and nullity of $(V; \Omega)$ by
			\begin{equation} \label{eq:higher.dim.index.lower.h}
				\ind_{A_{2\energyunit^{-1}\mfh_i}}(V; \Omega) + \nul_{A_{2\energyunit^{-1} \mfh_i}}(V; \Omega) \geq \lim_i (\ind_{E_{\eps_i,\mfh_i}}(u_i) + \nul_{E_{\eps_i,\mfh_i}}(u_i)).
			\end{equation}
			See Section \ref{sec:variations} for notation and Theorem \ref{theo:index.lower} and Corollary \ref{coro:index.lower} for more general results.
	\end{enumerate}
\end{theo}

The regularity theory developed for Theorem \ref{theo:higher.dim} (b) and (c) is an adaptation to arbitrary $\mfh$ of the Wang--Wei curvature estimates for transition layers when $\mfh \equiv 0$ \cite{WangWei} (see also \cite{WangWei2, ChodoshMantoulidis:multiplicity-one}). We are not ``generalizing'' the Wang--Wei estimates to all $\mfh$ because we only need and only prove estimates for so-called  multiplicity-one solutions. The multiplicity-one estimates we present here are quite direct versus the general curvature estimates of \cite{WangWei, WangWei2, ChodoshMantoulidis:multiplicity-one} that require substantial work. We have opted for a presentation that is as simple and self-contained as possible, so some of our smoothness assumptions are more stringent than necessary. We hope our streamlined exposition will help make the groundbreaking recent Wang--Wei regularity theory accessible to a wider audience.

We list a number of interesting future directions and unresolved questions in Section \ref{sec:open}. We encourage the reader to refer to Remarks \ref{rema:index.upper.known}, \ref{rema:mult.one}, \ref{rema:index.lower.known} for additional context on our results and how they fit within the literature. This work has been partially influenced and motivated by the enormous recent advances of Almgren--Pitts min-max theory, which the min-max theory of \eqref{eq:ac.energy.h} has been tracking in a somewhat parallel fashion. Since Almgren--Pitts theory is not so relevant to this work, we will only list a small number of results that we deem most relevant: \cite{KetoverLiokumovich:curves, MarquesNeves:index, MarquesNeves:multiplicity-one, Zhou, ZhouZhu:pmc, ZhouZhu:cmc}.

\textbf{Acknowledgments} The author would like to acknowledge Constante Bellettini, Otis Chodosh, and Xin Zhou for helpful conversations on constant mean curvature hypersurfaces. The author was supported in part by NSG Grant No. DMS-1905165/2050120/2147521.

\section{Variations of $A_\mfh$, $E_{\eps,\mfh}$}

\label{sec:variations}

\subsection{The $A_\mfh$ functional} \label{sec:variations.h}

Suppose that $\Sigma^{n-1}$ is a closed hypersurface in a closed Riemannian manifold $(M^n, g)$. The first variation formula gives the first order rate of change of the $(n-1)$-dimensional area $A[\cdot]$ of $\Sigma$ if $\Sigma$ is deformed in the direction given by an ambient vector field. Let $\bm{X}$ be a $C^1$ vector field on $M$ whose flow is given by $\Phi^t : M \to M$. The first variation of the area of $\Sigma$ along $\bm{X}$ is
\begin{equation} \label{eq:first.variation.smooth}
	\delta A[\Sigma]\{\bm{X}\} := \left[ \tfrac{d}{dt} A[(\Phi^t)_{\#} \Sigma] \right]_{t=0} = \int_\Sigma \divg_\Sigma \bm{X} \, d\mu_\Sigma.
\end{equation}
The second order rate of change of area along a $C^2$ vector field $\bm{X}$ with flow $\Phi^t : M \to M$ is given by the second variation:
\begin{align} \label{eq:second.variation.smooth}
	\delta^2 A[\Sigma]\{\bm{X}, \bm{X}\} 
		& := \left[ \tfrac{d^2}{dt^2} A[(\Phi^t)_{\#} \Sigma] \right]_{t=0} \nonumber \\
		& = \int_\Sigma \big[ \divg_\Sigma \nabla_{\bm{X}} \bm{X} + (\divg_\Sigma \bm{X})^2 + |\nabla_\Sigma^\perp \bm{X}|^2 \nonumber \\
		& \qquad - \sum_{i,j=1}^{n-1} (\nabla_{\bm{\tau}_i} \bm{X} \cdot \bm{\tau}_j) (\nabla_{\bm{\tau}_j} \bm{X} \cdot \bm{\tau}_i) - \sum_{i=1}^{n-1} \riem(\bm{X}, \bm{\tau}_i, \bm{\tau}_i, \bm{X}) \big] \, d\mu_\Sigma.
\end{align}
In \eqref{eq:second.variation.smooth}, $(\bm{\tau}_i)_{i=1,\ldots,n-1}$ gives an orthonormal frame for $T\Sigma$ at each point, $\riem(\bm{X}, \bm{\tau}_i, \bm{\tau}_i, \bm{X})$ is the sectional curvature (suitably scaled) of $(M, g)$ along $\bm{X} \wedge \bm{\tau}_i$, and $\nabla_\Sigma^\perp \bm{X}$ denotes the orthogonal component of $\nabla_\Sigma \bm{X}$. See \cite{Simon83} for a derivation of these identities in the Euclidean setting; the Riemannian modifications are straightforward. 

Note that $\delta A[\Sigma]\{\bm{X}\}$ depends only on $\bm{X}|_\Sigma$, while $\delta^2 A[\Sigma]\{\bm{X}, \bm{X}\}$ also depends on the behavior of $\bm{X}$ off $\Sigma$ (by virtue of the $\divg_\Sigma \nabla_{\bm{X}} \bm{X}$ term).

A hypersurface $\Sigma$ is said to be a critical point of the area functional if the first order rate of change $\delta A[\Sigma]\{\bm{X}\}$ is zero for all $C^1$ vector fields $\bm{X}$. Using the formula $\bm{H} = - (\divg_\Sigma \bm{n})\bm{n}$ in \eqref{eq:first.variation.smooth}, and integration by parts on the tangential component of $\bm{X}$, shows that the criticality condition is equivalent to $\bm{H} \equiv \bm{0}$ on $\Sigma$. For critical points, the second variation reduces to:
\begin{align} \label{eq:second.variation.smooth.critical}
	\delta^2 A[\Sigma]\{\bm{X}, \bm{X}\} 
		& = \int_\Sigma \big[ (\divg_\Sigma \bm{X})^2 + |\nabla_\Sigma^\perp \bm{X}|^2 \nonumber \\
		& \qquad - \sum_{i,j=1}^{n-1} (\nabla_{\bm{\tau}_i} \bm{X} \cdot \bm{\tau}_j) (\nabla_{\bm{\tau}_j} \bm{X} \cdot \bm{\tau}_i) - \sum_{i=1}^{n-1} \riem(\bm{X}, \bm{\tau}_i, \bm{\tau}_i, \bm{X}) \big] \, d\mu_\Sigma.
\end{align}
Note how, when $\Sigma$ is a critical point, $\delta^2 A[\Sigma]\{\bm{X}, \bm{X}\}$ only depends on $\bm{X}|_\Sigma$ and makes sense for $C^1$ (rather than $C^2$) vector fields $\bm{X}$.

We now consider the more general $\mfh$-area functional (or, the prescribed mean curvature $\mfh$ functional), $A_\mfh[\cdot]$, where $\mfh : M \to \RR$ is a fixed ambient function. We also require that $\Sigma$ bounds a domain $\Omega$. We set:
\[ A_\mfh[\Sigma; \Omega] := A[\Sigma] - \int_\Omega \mfh \, d\mu_g. \]
Then, for any vector field $C^1$ vector field $\bm{X}$ with flow $\Phi^t : M \to M$, the first variation of $A_\mfh[\cdot]$ is easily derived from \eqref{eq:first.variation.smooth} and the divergence theorem to be:
\begin{equation} \label{eq:first.variation.h.smooth}
	\delta A_\mfh[\Sigma; \Omega]\{\bm{X}\} := \left[ \tfrac{d}{dt} A_\mfh[(\Phi^t)_{\#} \Sigma; (\Phi^t)_{\#} \Omega] \right]_{t=0} = \int_\Sigma \divg_\Sigma \bm{X} \, d\mu_\Sigma - \int_\Sigma \mfh \bm{X} \cdot \bm{n} \, d\mu_\Sigma,
\end{equation}
where $\bm{n}$ is the unit normal to $\Sigma$ that points outside of $\Omega$. Despite presence of the bulk term in the definition of $A_\mfh$, we see that \eqref{eq:first.variation.h.smooth} also only depends on $\bm{X}|_\Sigma$, like in \eqref{eq:first.variation.smooth}.

As before, we say that $(\Sigma; \Omega)$ is a critical point of $A_\mfh[\cdot]$ if $\delta A_\mfh[\Sigma; \Omega]\{\bm{X}\} = 0$ for all vector fields $\bm{X}$. An integration by parts and the use of $\bm{H} = - (\divg_\Sigma \bm{n}) \bm{n}$ in \eqref{eq:first.variation.h.smooth} shows that criticality is now equivalent to the mean curvature vector of $\Sigma$ satisfying $\bm{H}= -(\mfh|_\Sigma) \bm{n}$. 

We state the second variation of $A_\mfh$ for critical points $(\Sigma; \Omega)$. If $\bm{X}$ denotes a $C^1$ vector field with $\bm{X}|_\Sigma \perp \Sigma$ and flow $\Phi^t : M\to M$, then the second variation of $A_\mfh$ along $\bm{X}$ is:
\begin{align} \label{eq:second.variation.h.smooth}
	\delta^2 A_\mfh[\Sigma; \Omega]\{\bm{X}, \bm{X}\} 
		& := \left[ \tfrac{d^2}{dt^2} A_\mfh[(\Phi^t)_{\#} \Sigma; (\Phi^t)_{\#} \Omega] \right]_{t=0} \nonumber \\
		& = \int_\Sigma \big[ (\divg_\Sigma \bm{X})^2 + |\nabla_\Sigma^\perp \bm{X}|^2 \nonumber \\
		& \qquad - \sum_{i,j=1}^{n-1} (\nabla_{\bm{\tau}_i} \bm{X} \cdot \bm{\tau}_j) (\nabla_{\bm{\tau}_j} \bm{X} \cdot \bm{\tau}_i) - \sum_{i=1}^{n-1} \riem(\bm{X}, \bm{\tau}_i, \bm{\tau}_i, \bm{X}) \nonumber \\
		& \qquad - (\nabla_{\bm{X}} \mfh)(\bm{X} \cdot \bm{n}) - \mfh (\bm{X} \cdot \bm{n}) \divg_\Sigma \bm{X} \big] \, d\mu_\Sigma.
\end{align}
This follows from \eqref{eq:second.variation.smooth} and the derivative of the flux term (i.e., the $\mfh$ term) in \eqref{eq:first.variation.h.smooth}. Note that we are restricting to $\bm{X}$ that are normal to $\Sigma$, but this is done without loss of generality: the tangential component of $\bm{X}$ only flows $\Sigma$ by self-diffeomorphisms so it has no effect on the area of $\Sigma$ or on the bulk integral in the definition of $A_\mfh$. It is also convenient to rewrite \eqref{eq:second.variation.h.smooth} in scalar notation, where we write $\bm{X} = f \bm{n}$ for some scalar valued function $f : \Sigma \to \RR$:
\begin{align} \label{eq:second.variation.h.smooth.scalar}
	\delta^2 A_\mfh[\Sigma; \Omega]\{f \bm{n}, f \bm{n}\} 
		& = \int_\Sigma \big[ |\nabla_\Sigma f|^2 - (|\sff_\Gamma|^2 + \ricc(\bm{n}, \bm{n}) + \tfrac{\partial}{\partial \bm{n}} \mfh) f^2\big] \, d\mu_\Sigma.
\end{align}
Note that \eqref{eq:second.variation.h.smooth.scalar} is the bilinear form associated with an elliptic operator on $\Sigma$. Since $\Sigma$ is closed, this operator naturally comes with a spectrum, which can be recovered by the well-known min-max characterization. Namely, for each $\ell \in \NN$, the $\ell$-th eigenvalue of $\delta^2 A_\mfh[\Sigma; \Omega]$ is given by:
\begin{multline} \label{eq:spectrum.h.smooth}
	\lambda_\ell(\delta^2 A_\mfh[\Sigma; \Omega]) := \inf \Big\{ \max \Big\{ \frac{\delta^2 A_\mfh[\Sigma; \Omega]\{\bm{X}, \bm{X}\}}{\int_\Sigma |\bm{X}|^2 \, d\mu_\Sigma} : \bm{X} \in F \setminus \{0\} \Big\} \\
		: F \text{ is an } \ell\text{-dimensional subspace of } C^1(\Sigma; N\Sigma) \Big\}.
\end{multline}
Here, $N\Sigma$ denotes the normal bundle of $\Sigma$ in $M$. Given \eqref{eq:spectrum.h.smooth}, one defines the $A_\mfh$ Morse index and nullity of $(\Sigma; \Omega)$ as follows:
\begin{align}
	\ind_{A_\mfh}(\Sigma; \Omega) & := \# \{ \ell \in \NN : \lambda_\ell(\delta^2 A_\mfh[\Sigma; \Omega]) < 0 \}, \label{eq:index.h.smooth} \\
	\nul_{A_\mfh}(\Sigma; \Omega) & := \# \{ \ell \in \NN : \lambda_\ell(\delta^2 A_\mfh[\Sigma; \Omega]) = 0 \}. \label{eq:nullity.h.smooth}
\end{align}
When $\mfh \equiv 0$, $\Omega$ is irrelevant, and we simply denote these quantities by $\ind(\Sigma)$, $\nul(\Sigma)$.

For the purposes of Section \ref{sec:index.upper}, we need to extend these notions to the nonsmooth setting of \cite{HutchinsonTonegawa00, RogerTonegawa08}. Our hypersurface $\Sigma^{n-1}$ will be replaced by an integral $(n-1)$-varifold $V$ (i.e., a countably rectifiable set with a.e. integer density), and the domain $\Omega$ enclosed by $\Sigma$ will get replaced by a Caccioppoli set (i.e., a set of finite perimeter), still labeled $\Omega$. We point the reader to \cite{Simon83} for a discussion of these objects from geometric measure theory.

If $V$ is an integral $(n-1)$-varifold and $\Omega$ is a Caccioppoli set, we define $A_\mfh[V; \Omega]$ as:
\[ A_\mfh[V; \Omega] := \Vert V \Vert(M) - \int_\Omega \mfh \, d\mu_g. \]
Then, the first variation of $A_\mfh[V; \Omega]$ by an ambient $C^1$ vector field $\bm{X}$ with flow $\Phi^t : M \to M$ is given by:
\begin{equation} \label{eq:first.variation.h.sing}
	\delta A_\mfh[V; \Omega]\{\bm{X}\} := \left[ \tfrac{d}{dt} A_\mfh[(\Phi^t)_{\#} V; (\Phi^t)_{\#} \Omega] \right]_{t=0} = \int \divg_V \bm{X} \, d\Vert V \Vert - \int \mfh \bm{X} \cdot \bm{n} \, d(\cH_g^{n-1} \restr \partial^* \Omega).
\end{equation}
A pair $(V; \Omega)$ is said to be a critical point for $A_\mfh$ if $\delta A_\mfh[V; \Omega]\{\bm{X}\} = 0$ for all $C^1$ vector fields $\bm{X}$ on $M$. The relevance of this definition to our work is that, by \cite{HutchinsonTonegawa00, RogerTonegawa08}: 
\begin{equation} \label{eq:hutchinson.tonegawa.critical}
\begin{aligned}
	\text{limiting objects } (V; \Omega) \text{ that come from solutions } (u_i,  \eps_i) \text{ of } \eqref{eq:ac.pde.h} \text{ with } \eps_i \to 0 \\
	\text{ and uniform bounds on } E_{\eps_i,\mfh}[u_i] 	\text{ are critical points of } A_{2\energyunit^{-1} \mfh}[V; \Omega].
\end{aligned}
\end{equation}
For critical points $(V; \Omega)$ of $A_\mfh[\cdot]$, and vector fields that are compactly supported away from the singular part of $V$, the second variation along a $C^1$ vector field $\bm{X}$ which is $\perp$ to $V$ and has flow $\Phi^t : M \to M$ is:
\begin{align} \label{eq:second.variation.h.thin.sing}
	\delta^2 A_\mfh[V; \Omega]\{\bm{X}, \bm{X}\} 
		& := \left[ \tfrac{d^2}{dt^2} A_\mfh[(\Phi^t)_{\#} V; (\Phi^t)_{\#} \Omega] \right]_{t=0} \nonumber \\
		& = 	\int \big[ (\divg_V \bm{X})^2 + |\nabla_V^\perp \bm{X}|^2 \nonumber \\
		& \qquad - \sum_{i,j=1}^{n-1} (\nabla_{\bm{\tau}_i} \bm{X} \cdot \bm{\tau}_j) (\nabla_{\bm{\tau}_j} \bm{X} \cdot \bm{\tau}_i) - \sum_{i=1}^{n-1} \riem(\bm{X}, \bm{\tau}_i, \bm{\tau}_i, \bm{X}) \big] \, d\Vert V \Vert \nonumber \\
		& - \int \big[ (\nabla_{\bm{X}} \mfh)(\bm{X} \cdot \bm{n}) + \mfh (\bm{X} \cdot \bm{n}) \divg_\Sigma \bm{X} \big] \, d(\cH^{n-1}_g \restr \partial^* \Omega).
\end{align}
By analogy with \eqref{eq:spectrum.h.smooth}, \eqref{eq:index.h.smooth}, \eqref{eq:nullity.h.smooth} we define, for any open $\cO \subset M \setminus \sing V$ and $\Sigma := \reg V$:
\begin{multline*}
	\lambda_\ell(\delta^2 A_\mfh[V; \Omega]; \cO) := \inf \Big\{ \max \Big\{ \frac{\delta^2 A_\mfh[\Sigma; \Omega]\{\bm{X}, \bm{X}\}}{\int |\bm{X}|^2 \, d\Vert V \Vert} : \bm{X} \in F \setminus \{\bm{0}\} \Big\} \\
	: F \text{ is an } \ell\text{-dimensional subspace of } C^1_c(\Sigma \cap \cO; N(\Sigma \cap \cO)) \Big\}.
\end{multline*}
and
\begin{align} 
	\lambda_\ell(\delta^2 A_\mfh[V; \Omega]) & := \inf \Big\{ \lambda_\ell(\delta^2 A_\mfh[V; \Omega]; \cO) : \cO \subset M \setminus \sing V \text{ is open} \Big\}, \label{eq:spectrum.h.sing} \\
	\ind_{A_\mfh}(V; \Omega) & := \# \{ \ell \in \NN : \lambda_\ell(\delta^2 A_\mfh[V; \Omega]) < 0 \}, \label{eq:index.h.sing} \\
	\nul_{A_\mfh}(V; \Omega) & := \# \{ \ell \in \NN : \lambda_\ell(\delta^2 A_\mfh[V; \Omega]) = 0 \}.\label{eq:nullity.h.sing}
\end{align}

\begin{rema} \label{rema:stability.sing}
	This approach to measuring the linear stability of $(V; \Omega)$ restricts to variations $\bm{X}$ that \emph{fix} the singular set of $V$. This is how one usually defines the index of non-smooth $V$. Are we potentially underestimating the index by only allowing deformations that fit $\sing V$? This depends on the ``size'' of $\sing V$.
	\begin{enumerate}
		\item \label{rema:stability.sing.small} When $\cH^{n-3}_g(\sing V) < \infty$, a cutoff argument of Federer--Ziemer (\cite[p. 89]{Wickramasekera:cutoff}) shows that the stability of $V$ is accurately captured by restricting to variations which are compactly supported away from $\sing V$.
		\item \label{rema:stability.sing.large} When the cutoff argument above fails (e.g., when $\cH^{n-3}_g(\sing V) = \infty$) we are legitimately in danger of underestimating the index. 
	\end{enumerate}
	Here is what is known about whether we fall under \eqref{rema:stability.sing.small} or \eqref{rema:stability.sing.large} when $V$ occurs as the limit of $(u_i, \eps_i, \mfh_i)$:
	\begin{enumerate}[resume]
		\item \label{rema:stability.sing.zero.h} When $n \geq 3$, $\mfh_i \equiv 0$ and we have uniform bounds on $\ind_{E_{\eps_i, 0}}(u_i)$, $\sing V$ has ambient codimension $\geq 8$ (and is empty for $n = 3, \ldots, 7$) and we are thus in the situation of the first bullet point.  This was established in \cite{TonegawaWickramasekera12, Guaraco, Hiesmayr}, crucially relying on the regularity theory for stable integral varifolds of codimension-1 (\cite{Wickramasekera14}). 
		\item \label{rema:stability.sing.nonzero.h} When $n \geq 3$ and we have uniform bounds on $\ind_{E_{\eps_i, \mfh_i}}(u_i)$ and  mild non-negativity on $\mfh_i$, it was recently shown in \cite{BellettiniWickramasekera:ac} (crucially relying on regularity theory for stable CMC integral varifolds of codimension-1 from \cite{BellettiniChodoshWickramasekera:cmc, BellettiniWickramasekera:cmc}) that $\sing V$ consists of two portions:
			\begin{itemize}
				\item the excisable codimension $\geq 8$ portion that was already present for $\mfh \equiv 0$; and
				\item a portion that consists of ``geometric'' singularities where multiplicity-one sheets of $V$ touch other multiplicity-one sheets of $V$ or other minimal (even-multiplicity) sheets of $V$.
			\end{itemize}
			The latter portion is contained in a countable union of ambient codimension-2 submanifolds, so it may or may not break the finiteness of $\cH^{n-3}_g(\sing V)$. It is an interesting problem to understand the finer structure of the geometric singular set. 
		\item \label{rema:stability.sing.n2} When $n=2$, the limit $V$ has a singular set that consists of isolated points (unless $\ind_{E_{\eps_i, \mfh_i}}(u_i)$ is identically zero, in which case $\sing V$ is empty; see \cite{Tonegawa05}). When $\sing V$ is nonempty, the stability of $\reg V$ does not guarantee the stability of $V$. This is an interesting situation that is to be addressed in separate work.
	\end{enumerate}
\end{rema}

\subsection{The $E_{\eps,\mfh}$ functional}

\label{sec:variations.ac}

Let $(M^n, g)$ be a closed Riemannian manifold. It is easy to see that the first second variation of the energy functional in \eqref{eq:ac.energy.h} along a direction $v \in C^1(M)$ is 
\begin{equation} \label{eq:first.variation.ac}
	\delta E_{\eps,\mfh}[u] \{ v \} := \left[ \tfrac{d}{dt} E_{\eps,\mfh}[u+tv] \right]_{t=0} = \int_M \big[ \eps (\nabla u \cdot \nabla v) + W'(u) v + \mfh v \big] \, d\mu_g,
\end{equation}
and, for critical points $u$ of $E_{\eps,\mfh}$, the second variation along a direction $v \in C^1(M)$ is
\begin{equation} \label{eq:second.variation.ac}
	\delta^2 E_{\eps,\mfh}[u] \{ v, v \} := \left[ \tfrac{d^2}{dt^2} E_{\eps,\mfh}[u+tv] \right]_{t=0} = \int_M \big[ \eps |\nabla v|^2 + W''(u) v^2 \big] \, d\mu_g.
\end{equation}
By analogy with \eqref{eq:spectrum.h.smooth}, \eqref{eq:index.h.smooth}, \eqref{eq:nullity.h.smooth} we define:
\begin{multline} \label{eq:spectrum.ac}
	\lambda_\ell(\delta^2 E_{\eps,\mfh}[u]) := \inf \Big\{ \max \Big\{ \frac{\delta^2 E_{\eps,\mfh}[u]\{v, v\}}{\int v^2 \, d\mu_g} : v \in F \setminus \{\bm{0}\} \Big\} \\
	: F \text{ is an } \ell\text{-dimensional subspace of } C^1(M) \Big\}.
\end{multline}
and
\begin{align} 
	\ind_{E_{\eps,\mfh}}(u) & := \# \{ \ell \in \NN : \lambda_\ell(\delta^2 E_{\eps,\mfh}[u]) < 0 \},\label{eq:index.ac} \\
	\nul_{E_{\eps,\mfh}}(u) & := \# \{ \ell \in \NN : \lambda_\ell(\delta^2 E_{\eps,\mfh}[u]) = 0 \}.\label{eq:nullity.ac}
\end{align}

\begin{rema} \label{rema:isoperimetric.index.h}
	The objects defined in \eqref{eq:spectrum.h.smooth}, \eqref{eq:index.h.smooth}, \eqref{eq:nullity.h.smooth}, \eqref{eq:spectrum.ac}, \eqref{eq:index.ac}, \eqref{eq:nullity.ac} measure the variational behavior of $(\Sigma; \Omega)$ and $u$ when critical points are constructed holding $\mfh$ \textit{fixed}. This is done, for instance, when one tries to construct hypersurfaces with prescribed mean curvature; see \cite{ZhouZhu:pmc, ZhouZhu:cmc}. In this approach, we are not interested (and do not control) the volume enclosed by $\Omega$. Alternatively, one may wish to construct critical points by holding the enclosed volume of $\Omega$ fixed, or $\int_M u \, d\mu_g$ fixed, and instead giving up control on $\mfh$. See Section \ref{sec:open}.
\end{rema}

\section{Upper bounds for eigenvalues of $\delta^2 E_{\eps,\mfh}$ as $\eps \to 0$}

\label{sec:index.upper}

\begin{theo} \label{theo:index.upper}
	Let $(M^n, g)$ be a closed manifold. Consider a sequence of critical points $u_i$ to $E_{\eps_i, \mfh_i}$ with $\eps_i \to 0$ and $\Vert \mfh_i \Vert_{W^{2,p}(M)} + E_{\eps_i,\mfh_i}[u_i] \leq E$ for all $i$, for $p > n$. Let $V$ denote the limiting varifold and $\Omega$ denote the limiting domain of $(u_i, \eps_i)$, and $\mfh$ denote the limiting $\mfh_i$ after passing to a subsequence $\eps_{i'} \to 0$. Then, for any vector field $\bm{X}$ on $(M^n, g)$ supported away from the singular part of $(V; \Omega)$, we have
	\begin{equation} \label{eq:second.variation.ac.limit}
		\energyunit^{-1} \lim_{i' \to \infty} \delta^2 E_{\eps_{i'},\mfh_{i'}}[u] \{ \nabla u \cdot \bm{X}, \nabla u \cdot \bm{X} \} = \delta^2 A_{2\energyunit^{-1} \mfh}[V; \Omega]\{\bm{X}, \bm{X}\} + \int (\nabla_{\bm{n}} \bm{X} \cdot \bm{n})^2 \, d\Vert V \Vert.
	\end{equation}
	Moverover, for every $\ell \in \NN$,
	\begin{equation} \label{eq:index.upper}
		\lambda_\ell(A_{2\energyunit^{-1}\mfh}[V; \Omega]) \geq \lim_{i \to \infty} \eps_{i'}^{-1} \lambda_{\ell}(E_{\eps_{i'},\mfh_{i'}}[u]).
	\end{equation}
\end{theo}

\begin{rema} \label{rema:index.upper.known}
	Note:
	\begin{enumerate}
		\item Theorem \ref{theo:index.upper} and its corollary Theorem 1 (a) bound from above the index of the \textit{regular set} of the limiting $(V; \Omega)$ in terms of the indices of $(u_i, \eps_i, \mfh_i)$. This generalizes what was known for $\mfh_i \equiv 0$ (\cite{Gaspar}) to arbitrary $\mfh_i$.
		\item It is important to note that $\mfh_i \equiv 0$ has a decided advantage over $\mfh_i \not \equiv 0$ in that the singular set of $V$ is \textit{always} (when $n \geq 3$) of high-enough codimension in $(M^n, g)$ and thus does not contribute to the index of $V$; see \eqref{rema:stability.sing.zero.h} in Remark \ref{rema:stability.sing}. This makes the main theorem of \cite{Gaspar} applicable even across $\sing V$ (by \eqref{rema:stability.sing.small}  in Remark \ref{rema:stability.sing}), which is not true of Theorem \ref{theo:index.upper} when $\mfh_i \not \equiv 0$ and there are large geometric singular sets in the sense of \eqref{rema:stability.sing.nonzero.h} or \eqref{rema:stability.sing.n2} in Remark \ref{rema:stability.sing}.
		\item In accordance with \cite{White:generic-transversality}, one hopes \textit{generic} background metrics $g$ to have the property that limiting $(V; \Omega)$ have \textit{no} geometric singular set in the sense of Remark \ref{rema:stability.sing}. This will make Theorem \ref{theo:index.upper} \textit{generically} applicable across $\sing V$.
	\end{enumerate}
	 
\end{rema}

In order to prove our upper semicontinuity variational result for $(u_i, \eps_i, \mfh_i)$ as $\eps_i \to 0$, it will be convenient to rederive the first and second variations of $E_{\eps_i,\mfh_i}$ for a special class of variations, called \emph{inner variations}, which geometrically perturb the level sets of $u_i$, rather than analytically perturb $u_i$ as one does in full generality for \eqref{eq:first.variation.ac}, \eqref{eq:second.variation.ac}. This second method was used in this setting with the same goal in mind in \cite{Gaspar} assuming $\mfh \equiv 0$. We follow that same method in this proof. For simplicity of notation, we write $u$, $\mfh$, $\eps$ in place of $u_i$, $\mfh_i$, $\eps_i$.

\begin{proof}[Proof of Theorem \ref{theo:index.upper}]
	Denote $u^t := u \circ \Phi^{-t}$, where $\Phi^t : M \to M$, $t \in \RR$, denotes the flow of $\bm{X}$. By the change of variables formula,
	\begin{align} \label{eq:second.variation.ac.flow}
	E_{\eps,\mfh}[u^t] 
		& = \int_M \big[ \tfrac12 \eps |\nabla u^t|^2 + \eps^{-1} W(u^t) + \mfh u^t \big] \, d\mu_g \nonumber \\
		& = \int_M \big[ \tfrac12 \eps |(\nabla u^t) \circ \Phi^t|^2 + \eps^{-1} W(u) + (\mfh \circ \Phi^t) u \big] |J\Phi^t| \, d\mu_g.
	\end{align}
	It will be convenient to introduce some auxiliary notation,  following \cite{Gaspar}:
	\begin{align*}
	S_{\bm{X}}(\bm{Y}_1, \bm{Y}_2) & = \nabla_{\bm{Y}_1} \bm{X} \cdot \nabla_{\bm{Y}_2} \bm{X}, \\
	h_{\bm{X}}(\bm{Y}_1, \bm{Y}_2) & =  \nabla_{\bm{Y}_1} \bm{X} \cdot \bm{Y}_2 + \bm{Y}_1 \cdot \nabla_{\bm{Y}_2} \bm{X} = (\cL_{\bm{X}} g)(\bm{Y}_1, \bm{Y}_2). 
	\end{align*}
	As in \cite{Gaspar}:
	\begin{align}
		\left[ \tfrac{\partial}{\partial t} |J\Phi^t| \right]_{t=0} 
		& = \divg \bm{X}, \label{eq:second.variation.ac.ddt.jphi} \\
		\left[ \tfrac{\partial^2}{\partial t^2} |J\Phi^t| \right]_{t=0} 
		& = \divg (\nabla_{\bm{X}} \bm{X}) - \ricc(\bm{X}, \bm{X}) + \tr_g S_{\bm{X}} - \tfrac12 |h_{\bm{X}}|^2 + (\divg \bm{X})^2 \nonumber \\
		& = \divg (\nabla_{\bm{X}} \bm{X}) - \ricc(\bm{X}, \bm{X}) + |\nabla \bm{X}|^2 - \tfrac12 |\cL_{\bm{X}} g|^2 + (\divg \bm{X})^2, \label{eq:second.variation.ac.ddt2.jphi}
	\end{align}
	as well as:
	\begin{align}
	\left[ \tfrac{\partial}{\partial t} |(\nabla u^t) \circ \Phi^t|^2 \right]_{t=0}
		& = - 2 \nabla_{\nabla u} \bm{X} \cdot \nabla u, \label{eq:second.variation.ac.ddt.gradu} \\
	\left[ \tfrac{\partial^2}{\partial t^2} |(\nabla u^t) \circ \Phi^t|^2 \right]_{t=0}
		& = 2 \riem(\bm{X}, \nabla u, \nabla u, \bm{X}) - 2 \nabla_{\nabla u} \nabla_{\bm{X}} \bm{X} \cdot \nabla u \nonumber \\
		& \qquad + 2 |\nabla u \cdot \nabla_{\bullet} \bm{X}|^2 + 4 (\nabla_{\nabla_{\nabla u} \bm{X}} \bm{X} \cdot \nabla u)^2. \label{eq:second.variation.ac.ddt2.gradu}
	\end{align}
	We also clearly have:
	\begin{align}
		\left[ \tfrac{\partial}{\partial t} (\mfh \circ \Phi^t) \right]_{t=0} & = \nabla \mfh \cdot \bm{X}, \label{eq:second.variation.ac.ddt.mfh} \\
		\left[ \tfrac{\partial^2}{\partial t^2} (\mfh \circ \Phi^t) \right]_{t=0} & = \nabla_{\bm{X}} \nabla \mfh \cdot \bm{X} \label{eq:second.variation.ac.ddt2.mfh}.
	\end{align}
	From \eqref{eq:second.variation.ac.flow}, \eqref{eq:second.variation.ac.ddt.jphi}, \eqref{eq:second.variation.ac.ddt.gradu}, \eqref{eq:second.variation.ac.ddt.mfh}, we get:
	\begin{align} 
		\left[ \tfrac{d}{dt} E_{\eps,\mfh}[u^t]	 \right]_{t=0}
		& = \int_M \big[ - \eps (\nabla_{\nabla u} \bm{X} \cdot \nabla u) + (\nabla \mfh \cdot \bm{X}) u \nonumber \\
		& \qquad + (\tfrac12 \eps |\nabla u|^2 + \eps^{-1} W(u) + \mfh u) \divg \bm{X} \big] \, d\mu_g, \label{eq:second.variation.ac.ddt}
	\end{align}
	and from \eqref{eq:second.variation.ac.flow}, \eqref{eq:second.variation.ac.ddt2.jphi}, \eqref{eq:second.variation.ac.ddt2.gradu}, \eqref{eq:second.variation.ac.ddt2.mfh}, we find that
	\begin{align} \label{eq:second.variation.ac.ddt2}
		\left[ \tfrac{d^2}{dt^2} E_{\eps,\mfh}[u^t] \right]_{t=0}
		& = \int_M \Big[ \eps \Big( \riem(\bm{X}, \nabla u, \nabla u, \bm{X}) - \nabla_{\nabla u} \nabla_{\bm{X}} \bm{X} \cdot \nabla u \nonumber \\
		& \qquad \qquad + |\nabla u \cdot \nabla_{\bullet} \bm{X}|^2 + 2 (\nabla_{\nabla_{\nabla u} \bm{X}} \bm{X} \cdot \nabla u)^2 \Big) \nonumber \\
		& \qquad + (\nabla_{\bm{X}} \nabla \mfh \cdot \bm{X}) u + (\nabla_{\bm{X}} \mfh \cdot \nabla_{\bm{X}} \bm{X}) u \nonumber \\
		& \qquad + 2 \Big( - \eps (\nabla_{\nabla u} \bm{X} \cdot \nabla u) + (\nabla \mfh \cdot \bm{X}) u \Big) \divg \bm{X} \nonumber \\
		& \qquad + (\tfrac12 \eps |\nabla u|^2 + \eps^{-1} W(u) + \mfh u) \nonumber \\
		& \qquad \qquad \cdot \Big( \divg (\nabla_{\bm{X}} \bm{X}) - \ricc(\bm{X}, \bm{X}) + |\nabla \bm{X}|^2 - \tfrac12 |\cL_{\bm{X}} g|^2 + (\divg \bm{X})^2 \Big) \Big] \, d\mu_g.
	\end{align}
	Note how, for critical points, \eqref{eq:second.variation.ac.ddt2} reduces to:
	\begin{align} \label{eq:second.variation.ac.ddt2.critical}
		\left[ \tfrac{d^2}{dt^2} E_{\eps,\mfh}[u^t] \right]_{t=0}
		& = \int_M \Big[ \eps \Big( \riem(\bm{X}, \nabla u, \nabla u, \bm{X}) + |\nabla u \cdot \nabla_{\bullet} \bm{X}|^2 + 2 (\nabla_{\nabla_{\nabla u} \bm{X}} \bm{X} \cdot \nabla u)^2 \Big) \nonumber \\
		& \qquad + (\nabla_{\bm{X}} \nabla \mfh \cdot \bm{X}) u  + 2 \Big( - \eps (\nabla_{\nabla u} \bm{X} \cdot \nabla u) + (\nabla \mfh \cdot \bm{X}) u \Big) \divg \bm{X} \nonumber \\
		& \qquad + (\tfrac12 \eps |\nabla u|^2 + \eps^{-1} W(u) + \mfh u) \nonumber \\
		& \qquad \qquad \cdot \Big( - \ricc(\bm{X}, \bm{X}) + |\nabla \bm{X}|^2 - \tfrac12 |\cL_{\bm{X}} g|^2 + (\divg \bm{X})^2 \Big) \Big] \, d\mu_g.
	\end{align}
	Letting $\eps \to 0$ in \eqref{eq:second.variation.ac.ddt2.critical}, invoking \cite{HutchinsonTonegawa00, RogerTonegawa08}, passing to a subsequence accordingly (though still denoting by $\eps \to 0$ for simplicity of notation), and denoting the subsequential limiting varifold by $V$ and the limiting enclosed domain (where $u \to -1$) by $\Omega$ and its outward pointing unit normal by $\bm{n}$:
	\begin{align*}
		& \energyunit^{-1} \lim_{\eps \to 0} \left[ \tfrac{d^2}{dt^2} E_{\eps,\mfh}[u^t] \right]_{t=0} \\
		& \qquad = \int \Big[ | \bm{n} \cdot  \nabla_{\bullet} \bm{X}|^2 + 2 (\nabla_{\nabla_{\bm{n}} \bm{X}} \bm{X} \cdot \bm{n})^2 - 2 (\nabla_{\bm{n}} \bm{X} \cdot \bm{n}) \divg \bm{X} \\
		& \qquad \qquad + |\nabla \bm{X}|^2 - \tfrac12 |\cL_{\bm{X}} g|^2 + (\divg \bm{X})^2 - \tr_V \riem(\bm{X}, \cdot, \cdot, \bm{X}) \Big] \, d\Vert V \Vert \\
		& \qquad - \energyunit^{-1} \int_\Omega \Big[ \mfh \Big( - \ricc(\bm{X}, \bm{X}) + |\nabla \bm{X}|^2 - \tfrac12 |\cL_{\bm{X}} g|^2 + (\divg \bm{X})^2 \Big) \\
		& \qquad \qquad \qquad + (\nabla_{\bm{X}} \nabla \mfh \cdot \bm{X}) + 2 (\nabla \mfh \cdot \bm{X}) \divg \bm{X} \Big] \, d\mu_g  \\
		& \qquad + \energyunit^{-1} \int_{M \setminus \Omega} \mfh \Big( - \ricc(\bm{X}, \bm{X}) + |\nabla \bm{X}|^2 - \tfrac12 |\cL_{\bm{X}} g|^2 + (\divg \bm{X})^2 \Big) \\
		& \qquad \qquad \qquad + (\nabla_{\bm{X}} \nabla \mfh \cdot \bm{X}) + 2 (\nabla \mfh \cdot \bm{X}) \divg \bm{X} \Big] \, d\mu_g.
	\end{align*}
	Note that
	\[ |\bm{n} \cdot \nabla_{\bullet} \bm{X}|^2 =  (\nabla_{\bm{n}} \bm{X} \cdot \bm{n})^2 + |\nabla_{\partial \Omega}^\perp \bm{X}|^2, \]
	\[ 2 (\nabla_{\nabla_{\bm{n}} \bm{X}} \bm{X} \cdot \bm{n})^2 + |\nabla \bm{X}|^2 - \tfrac12 |\cL_{\bm{X}} g|^2 = (\nabla_{\bm{n}} \bm{X} \cdot \bm{n})^2 - \sum_{i,j=1}^{n-1} (\nabla_{\bm{\tau}_i} \bm{X} \cdot  \bm{\tau}_j) (\nabla_{\bm{\tau}_j} \bm{X} \cdot \bm{\tau}_i), \]
	where $(\bm{\tau}_i)_{i=1,\ldots,n-1}$ is an orthonormal basis for the tangent space at a.e. point of $V$. Thus,
	\begin{align*}
		& \energyunit^{-1} \lim_{\eps \to 0} \left[ \tfrac{d^2}{dt^2} E_{\eps,\mfh}[u^t] \right]_{t=0} \\
		& \qquad = \int \Big[ 2 (\nabla_{\bm{n}} \bm{X} \cdot  \bm{n})^2 + |\nabla_V^\perp \bm{X}|^2 - 2 (\nabla_{\bm{n}} \bm{X} \cdot \bm{n}) \divg X \\
		& \qquad \qquad - \sum_{i,j=1}^{n-1}  (\nabla_{\bm{\tau}_i} \bm{X} \cdot \bm{\tau}_j) (\nabla_{\bm{\tau}_j} \bm{X} \cdot \bm{\tau}_i) + (\divg \bm{X})^2 - \tr_V \riem(\bm{X}, \cdot, \cdot, \bm{X}) \Big] \, d\Vert V \Vert \\
		& \qquad - \energyunit^{-1} \int_\Omega \Big[ \mfh \Big( - \ricc(\bm{X}, \bm{X}) + |\nabla \bm{X}|^2 - \tfrac12 |\cL_{\bm{X}} g|^2 + (\divg \bm{X})^2 \Big) \\
		& \qquad \qquad \qquad + (\nabla_{\bm{X}} \nabla \mfh \cdot \bm{X}) + 2 (\nabla \mfh \cdot \bm{X}) \divg \bm{X} \Big] \, d\mu_g  \\
		& \qquad + \energyunit^{-1} \int_{M \setminus \Omega} \mfh \Big( - \ricc(\bm{X}, \bm{X}) + |\nabla \bm{X}|^2 - \tfrac12 |\cL_{\bm{X}} g|^2 + (\divg \bm{X})^2 \Big) \\
		& \qquad \qquad \qquad + (\nabla_{\bm{X}} \nabla \mfh \cdot \bm{X}) + 2 (\nabla \mfh \cdot \bm{X}) \divg \bm{X} \Big] \, d\mu_g \\
		& \qquad = \int \Big[ (\nabla_{\bm{n}} \bm{X} \cdot \bm{n})^2 + |\nabla_V^\perp \bm{X}|^2 + (\divg_V \bm{X})^2 \\
		& \qquad \qquad - \sum_{i,j=1}^{n-1} (\nabla_{\bm{\tau}_i} \bm{X} \cdot \bm{\tau}_j)(\nabla_{\bm{\tau}_j} \bm{X} \cdot \bm{\tau}_i) - \tr_{\partial \Omega} \riem(\bm{X}, \cdot, \cdot, \bm{X}) \Big] \, d\Vert V \Vert \\
		& \qquad - \energyunit^{-1} \int_\Omega \Big[ \mfh \Big( - \ricc(\bm{X}, \bm{X}) + |\nabla \bm{X}|^2 - \tfrac12 |\cL_{\bm{X}} g|^2 + (\divg \bm{X})^2 \Big) \\
		& \qquad \qquad \qquad + (\nabla_{\bm{X}} \nabla \mfh \cdot \bm{X}) + 2 (\nabla \mfh \cdot \bm{X}) \divg \bm{X} \Big] \, d\mu_g  \\
		& \qquad + \energyunit^{-1} \int_{M \setminus \Omega} \mfh \Big( - \ricc(\bm{X}, \bm{X}) + |\nabla \bm{X}|^2 - \tfrac12 |\cL_{\bm{X}} g|^2 + (\divg \bm{X})^2 \Big) \\
		& \qquad \qquad \qquad + (\nabla_{\bm{X}} \nabla \mfh \cdot \bm{X}) + 2 (\nabla \mfh \cdot \bm{X}) \divg \bm{X} \Big] \, d\mu_g.
	\end{align*}
	We recognize, from \eqref{eq:second.variation.smooth},
	\begin{align} \label{eq:second.variation.ac.area.term}
		& \int \Big[ (\nabla_{\bm{n}} \bm{X} \cdot \bm{n})^2 + |\nabla_V^\perp \bm{X}|^2 + (\divg_V \bm{X})^2 \nonumber \\
		& \qquad - \sum_{i,j=1}^{n-1} (\nabla_{\bm{\tau}_i} \bm{X} \cdot \bm{\tau}_j)(\nabla_{\bm{\tau}_j} \bm{X} \cdot \bm{\tau}_i) - \tr_{\partial \Omega} \riem(\bm{X}, \cdot, \cdot, \bm{X}) \Big] \, d\Vert V \Vert \nonumber \\
		& = \Big[ \tfrac{d^2}{dt^2} \Vert (\Phi^{-t})_{\#} V \Vert(M)\Big]_{t=0} - \int \divg_V \nabla_{\bm{X}} \bm{X} \, d\Vert V \Vert + \int (\nabla_{\bm{n}} \bm{X} \cdot \bm{n})^2 \, d\Vert V \Vert.
	\end{align}
	We also recognize, from \eqref{eq:second.variation.ac.ddt2.jphi},
	\begin{align} \label{eq:second.variation.ac.inner.vol.term}
		& \int_\Omega \Big[ \mfh \Big( - \ricc(\bm{X}, \bm{X}) + |\nabla \bm{X}|^2 - \tfrac12 |\cL_{\bm{X}} g|^2 + (\divg \bm{X})^2 \Big) \nonumber \\
		& \qquad + (\nabla_{\bm{X}} \nabla \mfh \cdot \bm{X}) + 2 (\nabla \mfh \cdot \bm{X}) \divg \bm{X} \Big] \, d\mu_g \nonumber \\		
		& = \Big[ \tfrac{d^2}{dt^2} \int_{\Phi^{-t}(\Omega)} \mfh \, d\mu_g \Big]_{t=0} - \int_\Omega \big[ \nabla \mfh \cdot \nabla_{\bm{X}} \bm{X} + \mfh \divg(\nabla_{\bm{X}} \bm{X}) \big] \, d\mu_g
	\end{align}
	and, similarly,
	\begin{align} \label{eq:second.variation.ac.outer.vol.term}
		& \int_{M \setminus \Omega} \Big[ \mfh \Big( - \ricc(\bm{X}, \bm{X}) + |\nabla \bm{X}|^2 - \tfrac12 |\cL_{\bm{X}} g|^2 + (\divg \bm{X})^2 \Big) \nonumber \\
		& \qquad + (\nabla_{\bm{X}} \nabla \mfh \cdot \bm{X}) + 2 (\nabla \mfh \cdot \bm{X}) \divg \bm{X} \Big] \, d\mu_g \nonumber \\		
		& = \Big[ \tfrac{d^2}{dt^2} \int_{\Phi^{-t}(M \setminus \Omega)} \mfh \, d\mu_g \Big]_{t=0} - \int_{M \setminus \Omega} \big[ \nabla \mfh \cdot \nabla_{\bm{X}} \bm{X} + \mfh \divg(\nabla_{\bm{X}} \bm{X}) \big] \, d\mu_g
	\end{align}	
	From \eqref{eq:second.variation.ac.area.term}, \eqref{eq:second.variation.ac.inner.vol.term}, \eqref{eq:second.variation.ac.outer.vol.term}, and integration by parts, we deduce that
	\begin{align*} 
		\energyunit^{-1} \lim_{\eps \to 0} \left[ \tfrac{d^2}{dt^2} E_{\eps,\mfh}[u^t] \right]_{t=0} 
		& = \Big[ \tfrac{d^2}{dt^2} \Big( \Vert (\Phi^{-t})_{\#} V \Vert(M) - 2 \energyunit^{-1} \int_{\Phi^{-t}(\Omega)} \mfh \, d\mu_g \Big) \Big]_{t=0} \nonumber \\
		& \qquad - \int \divg_V \nabla_{\bm{X}} \bm{X} \, d\Vert V \Vert - \int_\Omega \big[ \nabla \mfh \cdot \nabla_{\bm{X}} \bm{X} + \mfh \divg(\nabla_{\bm{X}} \bm{X}) \big] \, d\mu_g. \nonumber \\
		& \qquad + \int (\nabla_{\bm{n}} \bm{X} \cdot \bm{n})^2 \, d\Vert V \Vert.
	\end{align*}
	The first and second integrals of the right hand side cancel each other out since $(V; \Omega)$ is a critical point of $A_{2\energyunit^{-1}\mfh}$ by \cite{HutchinsonTonegawa00, RogerTonegawa08}; \eqref{eq:second.variation.ac.limit} follows.
	
	We proceed to the eigenvalue estimate in \eqref{eq:index.upper}. Write $\Sigma := \reg V$. Fix $\eta > 0$, and let $\cO \subset M \setminus \sing V$ be such that 
	\begin{equation} \label{eq:second.variation.ac.eigenvalue.approx.1}
		\lambda_\ell(\delta^2 A_{2\energyunit^{-1} \mfh}[V; \Omega]; \cO) \leq \lambda_\ell(\delta^2 A_{2\energyunit^{-1} \mfh}[V; \Omega]) + \eta.
	\end{equation}
	Let $F$ be an $\ell$-dimensional subspace of $C^1_c(\Sigma \cap \cO; N(\Sigma \cap \cO))$, chosen so that 
	\begin{equation} \label{eq:second.variation.ac.eigenvalue.approx.2} 
		\frac{\delta^2 A_{2\energyunit^{-1} \mfh}[V; \Omega]\{ \bm{X}, \bm{X} \}}{\int |\bm{X}|^2 \, d\Vert V \Vert} \leq \lambda_\ell(\delta^2 A_{2\energyunit^{-1} \mfh}[V; \Omega]; \cO) + \eta, \; \text{ for all } \bm{X} \in F \setminus \{\bm{0}\}.
	\end{equation}
	We can easily extend each of the vector fields $\bm{X} \in F$ to an ambient vector field supported away from the singular set of $(V; \Omega)$ and with $\nabla_{\bm{n}} \bm{X} = \bm{0}$ along $\Sigma \cap \cO$. It is easy to see that the linear mapping
	\[ F \ni \bm{X} \mapsto (\nabla u \cdot \bm{X}) \in L^2(M) \]
	is injective for sufficiently small $\eps$ (see, e.g., \cite[Section 4]{Gaspar}, for these details). It follows that 
	\[ \{ (\nabla u \cdot \bm{X}) : \bm{X} \in F \} \subset L^2(M) \]
	is $\ell$-dimensional. From \eqref{eq:second.variation.ac.limit}, \eqref{eq:second.variation.ac.eigenvalue.approx.1}, \eqref{eq:second.variation.ac.eigenvalue.approx.2}, and the convergence $\eps (\bm{X} \cdot \nabla u)^2 \, d\mu_g \weaklyto \energyunit |\bm{X}|^2 \, d \Vert V \Vert$ as $\eps \to 0$ (\cite{HutchinsonTonegawa00, RogerTonegawa08}), we have:
	\[ \lim_{\eps \to 0} \eps^{-1} \lambda_\ell(\delta^2 E_{\eps,\mfh}[u]) \leq \lambda_\ell(\delta^2 A_{2\energyunit^{-1} \mfh}[V; \Omega]) + 2 \eta. \]
	The result follows since $\eta > 0$ is arbitrary.
\end{proof}

\section{Multiplicity-one asymptotics: $\eps$, $\eps^2$, and $o(\eps^2)$}

\label{sec:mult.one}

The analogous lower semicontinuity variational results for $(u, \eps, \mfh)$ as $\eps \to 0$ are more subtle than their upper semicontinuity counterparts from Section \ref{sec:index.upper}, where we only needed to use the existence (\cite{HutchinsonTonegawa00, RogerTonegawa08}) and regularity (\cite{BellettiniWickramasekera:ac}) of the weak limit as $\eps \to 0$. For the lower semicontinuity, we need a sharp understanding of the regularity of $u$ near $\{u = 0\}$, \emph{before} taking the limit $\eps \to 0$. It was shown in \cite{HutchinsonTonegawa00} that in the $O(\eps)$-scale around ``most'' points of $\{ u = 0 \}$, $u$ looks approximately like the one-dimensional solution $\mbh : \RR \to (-1, 1)$ of
\begin{equation} \label{eq:heteroclinic}
	\mbh'' = W'(\mbh),  \; \mbh(0) = 0,
\end{equation}
even in the presence of higher multiplicity and/or of $\mfh$. For our purposes, we need to find the expansion of $u$ up to order $o(\eps^2)$:
\begin{equation} \label{eq:u.expansion}
	u(y, z) = \mbh(\eps^{-1} z) + \eps \big[ \cdots \big] + \eps^2 \big[ \cdots \big] + o(\eps^2),
\end{equation}
with $y \in \{ u = 0 \}$ being the interface coordinate and $z$ being the vertical coordinate off the interface. This is the necessary order of approximation in order to obtain the lower semicontinuity relations between $(u, \eps, \mfh)$ and the $\eps \to 0$ limit $(V; \Omega)$.

When $\mfh \equiv 0$, the $\eps \big[ \cdots \big]$-term was determined and exploited for regularity purposes in the foundational paper of \cite{WangWei}, even allowing for many sheets in $\{ u = 0 \}$ (i.e., in high multiplicity). This was further refined in \cite{ChodoshMantoulidis:multiplicity-one, WangWei2}.

The next term in the asymptotic expansion of $u$ encodes the interference between sheets of $\{ u = 0 \}$ and is $\eps^2 |\log \eps| \big[ \cdots \big]$, rather than $\eps^2 \big[ \cdots \big]$, in the presence of high multiplicity; see \cite{delPinoKowalczykWeiYang:interface}. This obstructs one's ability to relate the variational structure of $(u, \eps, \mfh)$ to that of $(V; \Omega)$. However, this generally doesn't occur unless the solutions $(u, \eps, \mfh)$ are extremely variationally unstable, in which case no lower semicontinuity result is to be expected of $(u, \eps, \mfh)$; see \cite{ChodoshMantoulidis:multiplicity-one, WangWei2}. So, to get a proper lower semicontinuity result, one needs to restrict to \emph{multiplicity-one}, as we are doing here, where we can indeed verify that the next term in the asymptotic expansion is $\eps^2 \big[ \cdots \big]$. When $\mfh \equiv 0$, the $O(\eps^2)$ behavior in multiplicity-one was determined and exploited in \cite{ChodoshMantoulidis:multiplicity-one}.

In this section, we deduce both of the $\eps \big[ \cdots \big]$ and $\eps^2 \big[ \cdots \big]$ asymptotics without the extra assumption of $\mfh \equiv 0$, in the case of multiplicity-one convergence.

\begin{rema} \label{rema:mult.one}
	Besides extend to $\mfh \not \equiv 0$, the results we present here simplify the corresponding multiplicity-one results in \cite{WangWei, WangWei2} and \cite{ChodoshMantoulidis:multiplicity-one}. We still follow the strategy of \cite{WangWei}. We \textit{do not}, however, pursue the higher multiplicity regularity question that was pursued in these papers. For geometric applications, in accordance with \cite{Mantoulidis, ZhouZhu:pmc, ZhouZhu:cmc}, and \ref{rema:intro.h.odd.sharp}-\ref{rema:intro.h.even.sharp} of the introduction, one generally expects (and must verify!) multiplicity-one convergence. See also Section \ref{sec:open}.
\end{rema}

Throughout the section, we work in a Riemannian manifold $(M^n, g)$, $n \geq 2$.  Our approximation results are purely local, so we need not assume $(M, g)$ to be closed or even complete, provided we take care to work away from its boundary. 

In what follows,\footnote{These are all the same assumptions as in \cite[Section 2.1]{ChodoshMantoulidis:multiplicity-one}, with slightly more regularity on the background metric $g$ to streamline the exposition, and of course the added single-sheeted assumption.} let us fix $\alpha \in (0, 1)$ and  assume that we're working inside a precompact open set $\cO \subset M$ where the ambient metric $g$ is $C^5$ close to Euclidean,
\begin{equation} \label{eq:mult.one.assumptions.0}
	\sum_{\ell=0}^5 |\partial^\ell(g_{ij} - \delta_{ij})| \leq \eta_0 \text{ on } \cO,
\end{equation}
for some small $\eta_0 > 0$. For the solution $(u, \eps, \mfh)$ of \eqref{eq:ac.pde.h}, we assume that for some $\eps_0$, $E_0 > 0$:
\begin{equation} \label{eq:mult.one.assumptions.i}
	\eps \leq \eps_0, \; |u| \leq E_0 \text{ on } \cO, \; (E_{\eps,\mfh} \restr \cO)[u] \leq E_0, \; \Vert \mfh \Vert_{C^{3,\alpha}(\cO)} \leq E_0,
\end{equation}
and for some $\beta_0 \in (0, 1)$, $c_0 > 0$:
\begin{equation} \label{eq:mult.one.assumptions.ii}
	\eps |\nabla u| \geq c_0^{-1} > 0 \text{ on } \cO \cap \{ |u| < 1-\beta_0 \},
\end{equation}
which forces $\nabla u \neq \bm{0}$ and thus all level sets in $\cO \cap \{ |u| < 1-\beta_0 \}$ to be smooth, as well as that $\cA = \nabla(\nabla u/|\nabla u|)$ satisfies:
\begin{equation} \label{eq:mult.one.assumptions.iii}
	|\cA| + \eps |\nabla \cA| + \eps^2 |\nabla^2 \cA| + \eps^3 |\nabla^3 \cA| \leq c_0 \text{ on } \cO \cap \{ |u| < 1-\beta_0 \};
\end{equation}
cf. \cite[(2.3)-(2.7)]{ChodoshMantoulidis:multiplicity-one} for estimates up to $\nabla^2 \cA$; the estimate on $\nabla^3 \cA$ follows using the additional regularity we are assuming on $g$ in \eqref{eq:mult.one.assumptions.0}. We assume that
\[ \Gamma := \cO \cap \{ u = 0 \} \]
is a \emph{connected} smooth submanifold and that (possibly after rescaling) we have well-defined Fermi coordinates
\[ (y, z) \in B_{20}^\Gamma \times (-2, 2) \subset \cO \]
off $\Gamma$, where $B_{20}^\Gamma$ indicates a fixed geodesic ball within $\Gamma$ that is diffeomorphic to a disk. In these Fermi coordinates, $g_z$ denotes the metric induced on the constant-$z$ hypersurfaces parallel to $\Gamma$, and $\sff_z$, $H_z$ denote their second fundamental form 2-tensor and mean curvature scalar with $\partial_z$ taken as the ``outward'' normal. With this convention,
\[ \sff_z(\bm{X}, \bm{Y}) = \nabla_{\bm{X}} \partial_z \cdot \bm{Y} \text{ for } \bm{X}, \bm{Y} \text{ tangent to } \Gamma, \; H_z = \tr_{g_z} \sff_z, \text{ and } \]
\[ \Delta = \Delta_{g_z} + H_z \partial_z + \partial_z^2. \]
We will write $\nabla_\Gamma$, $\Delta_\Gamma$, $\sff_\Gamma$, $H_\Gamma$, $\bm{n}_\Gamma$ in place of $\nabla_{g_0}$, $\Delta_{g_0}$, $\sff_0$, $H_0$, $\partial_z|_{\Gamma}$. Note that a geometric consequence of \eqref{eq:mult.one.assumptions.0}, \eqref{eq:mult.one.assumptions.ii}, \eqref{eq:mult.one.assumptions.iii} and the Riccati equation
\[ \cL_{\partial_z} \sff_z = \sff_z^2 - \riem_g(\cdot, \partial_z, \partial_z, \cdot) \]
is the following $C^3_\eps$ bound (i.e., $C^3$ bound in the $\eps$-scale) on the second fundamental forms $\sff_z$, $z \in (-2, 2)$:
\begin{equation} \label{eq:mult.one.sff.bounds}
	|\sff_z| + \eps |\nabla_{g_z} \sff_z| + \eps^2 |\nabla_{g_z}^2 \sff_z| + \eps^3 |\nabla_{g_z}^3 \sff_z| \leq c_1
\end{equation}
for some $c_1 = c_1(n, \eta_0, c_0)$. For any $f : B_{20}^\Gamma \to \RR$ (independent of $z$) we have:
\begin{equation} \label{eq:mult.one.ddt.connection}	
\begin{aligned} 
	\cL_{\partial_z} \nabla_{g_z} f & = - 2 \sff_z(\nabla_{g_z} f, \cdot), \\
	\cL_{\partial_z} \nabla_{g_z}^2 f & = - \nabla^{g_z}_{\nabla_{g_z} f} \sff_z, \\
	\cL_{\partial_z} \Delta_{g_z} f & = -2 \langle \sff_z, \nabla^2_{g_z} f \rangle_{g_z} - \langle \nabla_{g_z} H_z, \nabla_{g_z} f \rangle_{g_z}; 
\end{aligned}
\end{equation}
see \cite[Appendix A]{ChodoshMantoulidis:multiplicity-one}. Together, \eqref{eq:mult.one.sff.bounds}, \eqref{eq:mult.one.ddt.connection} culminate in the $\eps$-scale estimates:
\begin{equation} \label{eq:mult.one.connection.estimates}
	\sum_{k=0}^\ell \eps^k \Vert (\nabla_{g_z} - \nabla_{\Gamma}) f \Vert_{C^k(U)} \leq c_1 |z| \sum_{k=0}^\ell \eps^k \Vert \nabla_\Gamma f \Vert_{C^k(U)},
\end{equation}
for $\ell = 0$, $1$, $2$, $3$, any $U \subset B^\Gamma_{20}$, and a possibly larger $c_1 = c_1(n, \eta_0, c_0)$.

Fix $\delta_* \in (0, 1)$ throughout. We define the cut off heteroclinic $\overline \mbh : \RR \to (-1, 1)$ given by:\footnote{We use a wider cutoff than in \cite{WangWei, ChodoshMantoulidis:multiplicity-one}. One might also use the cutoff of \cite{WangWei, ChodoshMantoulidis:multiplicity-one} after analyzing the exponential decay rate of the auxiliary function $\mbi$ defined in \eqref{eq:mult.one.i.ode}. We do not pursue this here.}
\begin{equation} \label{eq:mbh.cutoff}
	\overline \mbh(t) := \chi(\eps^{\delta_*} t) \mbh(t) + (\operatorname{sign} t) (1 - \chi(\eps^{\delta_*} t)),
\end{equation}
where $\chi$ indicates a smooth cutoff function such that
\begin{equation} \label{eq:mbh.cutoff.chi}
	\chi(t) = 1 \text{ for } t \in (-1, 1), \; \support \chi \subset (-2, 2).
\end{equation}
The exponential asymptotics of $\mbh$ ($\pm 1 \mp A \exp(-\sqrt{W''(\pm 1)} |t|)$ as $t \to \pm \infty$) give:
\begin{equation} \label{eq:mbh.cutoff.estimates}
	|\overline \mbh '' - W'(\overline \mbh)|_{C^3(\RR)} = O(\eps^3).
\end{equation}
Throughout, all $O(\cdot)$-notation and $o(\cdot)$-notation will be used under the assumption that we're sending $\eps \to 0$. 

Using the implicit function theorem and the multiplicity-one condition, as in \cite[Proposition 9.1]{WangWei}, one can produce an auxiliary function $h : B_{20}^\Gamma \to \RR$ such that
\begin{equation} \label{eq:mult.one.h.initial.est} 
	\Vert h \Vert_{C^{3,\alpha}_\eps(B_{19}^\Gamma)} = o(\eps)
\end{equation}
and
\begin{equation} \label{eq:mult.one.phi.orthogonal}
	\int_{-2}^2 (u(y,z) - \overline \mbh(\eps^{-1}(z-h(y)))) \overline \mbh'(\eps^{-1}(z-h(y))) \, dz = 0
\end{equation}
for all $y \in B_{19}^\Gamma$. Throughout the paper, $C^{\ell}_\eps$ and $C^{\ell,\alpha}_\eps$ denote the standard $\eps$-scaled weighted Banach spaces whose norms are:
\begin{equation} \label{eq:mult.one.ckalpha.notation}
	\Vert f \Vert_{C^\ell_\eps} := \sum_{k=0}^\ell \eps^k \Vert f \Vert_{C^\ell}, \; \Vert f \Vert_{C^{\ell,\alpha}_\eps} := \Vert f \Vert_{C^{\ell}_\eps} + \eps^{\ell+\alpha} [\nabla^\ell f]_\alpha.
\end{equation}

Our goal is to get an expansion for $u - \overline{\mbh}$ in terms of $\eps$, $\eps^2$, and $o(\eps^2)$. To that end, we first compute the PDE satisfied by $\overline \mbh(\eps^{-1}(z-h(y)))$ in Fermi coordinates $(y, z)$ off $\Gamma$; cf. \cite[(2.18)]{ChodoshMantoulidis:multiplicity-one}. For simplicity, we write
\begin{align*}
	\overline \mbh_\eps(y,z) & := \overline \mbh(\eps^{-1}(z-h(y))), \\
	\overline \mbh_\eps'(y,z) & := \overline \mbh'(\eps^{-1}(z-h(y))), \\
	\overline \mbh_\eps''(y,z) & := \overline \mbh''(\eps^{-1}(z-h(y))), \\
	\overline \mbh_\eps'''(y,z) & := \overline \mbh'''(\eps^{-1}(z-h(y))),
\end{align*}
and
\[ \mathbf{I}_\eps := (-3\eps^{1-\delta_*}, 3\eps^{1-\delta_*}). \]
Note that, by \eqref{eq:mult.one.h.initial.est}, $\overline \mbh_\eps(y, z) = \pm 1$ for $z \not \in \mathbf{I}_\eps$. On $B_{19}^\Gamma \times (-1, 1)$:
\begin{align} 
	\eps^2 \Delta \overline \mbh_\eps
		& = \eps^2 (\Delta_{g_z} + H_z \partial_z + \partial_z^2) \overline \mbh_\eps \nonumber \\
		& = (1 + |\nabla_{g_z} h|^2) \overline \mbh_\eps'' + \eps (H_z - \Delta_{g_z} h) \overline \mbh_\eps'  \nonumber \\
		& = W'(\overline \mbh_\eps) + \eps (H_\Gamma - \Delta_\Gamma h) \overline \mbh_\eps' \nonumber \\
		& \qquad + \eps (H_z - H_\Gamma) \overline \mbh_\eps' - \eps(\Delta_{g_z} h - \Delta_\Gamma h) \overline \mbh_\eps' + |\nabla_{g_z} h|^2 \overline \mbh_\eps'' + (\overline{\mbh}_{\eps}'' - W'(\overline{\mbh}_\eps)). \nonumber
\end{align}
Taylor expanding in around $z=0$ (i.e., around $\Gamma$), using the Riccati equation, \eqref{eq:mult.one.sff.bounds},  \eqref{eq:mult.one.ddt.connection}, \eqref{eq:mbh.cutoff.estimates}, \eqref{eq:mult.one.h.initial.est} we deduce that on $B^\Gamma_{19} \times (-1, 1)$:
\begin{align} \label{eq:mult.one.mbh.pde}
	\eps^2 \Delta \overline \mbh_\eps
		& = W'(\overline \mbh_\eps) + \eps (H_\Gamma - \Delta_\Gamma h) \overline \mbh_\eps' - \eps (|\sff_\Gamma|^2 + \ricc(\bm{n}_\Gamma, \bm{n}_\Gamma)) z   \overline \mbh_\eps' \nonumber \\
		& \qquad + O_{1,0,\alpha,\eps}(\eps \nabla_\Gamma^2 h, \nabla_\Gamma h) z \overline \mbh_\eps' + (O_{1,0,\alpha,\eps}(\nabla_\Gamma h))^2 \overline \mbh_\eps'' + O_{1,0,\alpha,\epsilon}(\eps^3).
\end{align}
Throughout, $O_{1,0,\alpha,\eps}(\{f_j\}_j)$ denotes a term $\mathcal{R}$ that is bounded by
\begin{multline} \label{eq:mult.one.bigO.notation}
	|\cR| \leq C \sum_j |f_j|, \; \Vert \cR \Vert_{C^{0,\alpha}_\eps} \leq C \sum_j \Vert f_j \Vert_{C^{0,\alpha}_\eps}, \; |\eps \partial_{y_i} \cR| \leq C \sum_j |f_j| + |\eps \partial_{y_i} f_j| \\
	\Vert \eps \partial_{y_i} \cR \Vert_{C^{0,\alpha}_\eps} \leq C \sum_j \Vert f_j \Vert_{C^{0,\alpha}_{\eps}} + \Vert \eps f_j \Vert_{C^{0,\alpha}_{\eps}}
\end{multline}
in the domain in question, with $C > 0$ fixed as $\eps \to 0$. We emphasize that derivatives in \eqref{eq:mult.one.bigO.notation} are only taken along  directions $y_i$ parallel to $\Gamma$, $i = 1, \ldots, n-1$, (because we will sometimes wish to differentiate along $y_i$) and that the H\"older seminorms are standard ($\eps$-weighted) H\"older seminorms in both $y$ and $z$, as in \eqref{eq:mult.one.ckalpha.notation} (because we will  use Schauder theory). In what follows, \eqref{eq:mult.one.sff.bounds} gets used repeatedly though implicitly when obtaining $O_{1,0,\alpha,\eps}$ bounds.

Throughout this work, we will  frequently rely on the fact that
\begin{equation} \label{eq:mbh.times.polynomial}
	\sup_{z \in \RR} |z|^k |\mbh^{(\ell)}(z)| < \infty \text{ for all } k,  \ell \in \NN, \; \ell \geq 1,
\end{equation}
to control terms such as $z \overline \mbh_\eps'$; this estimate follows from the exponential decay of $\mbh^{(\ell)}$.

Following \cite{ChodoshMantoulidis:multiplicity-one}, we set
\[ \phi := u - \overline \mbh_\eps. \]
Together, \eqref{eq:ac.pde.h} and \eqref{eq:mult.one.mbh.pde} imply that, on $B_{19}^\Gamma \times (-1, 1)$:
\begin{align} 
	\eps^2 \Delta \phi 
		& = \eps \mfh + W'(u) - W'(\overline \mbh_\eps) \nonumber \\
		& \qquad - \eps (H_\Gamma - \Delta_\Gamma h) \overline \mbh_\eps' + \eps (|\sff_\Gamma|^2 + \ricc(\bm{n}_\Gamma, \bm{n}_\Gamma)) z \overline \mbh_\eps' \nonumber \\
		& \qquad + O_{1,0,\alpha,\eps}(\eps \nabla_\Gamma^2 h, \nabla_\Gamma h) \overline \mbh_\eps' + (O_{1,0,\alpha,\eps}(\nabla_\Gamma h))^2 \overline \mbh_\eps'' + O_{1,0,\alpha,\eps}(\eps^3) \nonumber \\
		& = \eps \mfh + W''(\overline \mbh_\eps) \phi + \tfrac12 W'''(\overline \mbh_\eps) \phi^2 +  (O_{1,0,\alpha,\eps}(\phi))^3 \nonumber \\
		& \qquad - \eps (H_\Gamma - \Delta_\Gamma h) \overline \mbh_\eps' + \eps (|\sff_\Gamma|^2 + \ricc(\bm{n}_\Gamma, \bm{n}_\Gamma)) z \overline \mbh_\eps' \nonumber \\
		& \qquad + O_{1,0,\alpha,\eps}(\eps \nabla_\Gamma^2 h,  \nabla_\Gamma h) \overline \mbh_\eps' + (O_{1,0,\alpha,\eps}(\nabla_\Gamma h))^2 \overline \mbh_\eps'' + O_{1,0,\alpha,\eps}(\eps^3), \nonumber
\end{align}
i.e.,
\begin{align} \label{eq:mult.one.phi.pde}
	& \eps^2 \Delta \phi - W''(\overline \mbh_\eps) \phi \nonumber \\
	& \qquad = \eps \mfh - \eps (H_\Gamma - \Delta_\Gamma h) \overline \mbh_\eps' \nonumber \\
	& \qquad + \eps (|\sff_\Gamma|^2 + \ricc(\bm{n}_\Gamma, \bm{n}_\Gamma)) z \overline \mbh_\eps' + \tfrac12 W'''(\overline \mbh_\eps) \phi^2 \nonumber \\
	& \qquad + (O_{1,0,\alpha,\eps}(\phi))^3 + O_{1,0,\alpha,\eps}(\eps \nabla_\Gamma^2 h, \nabla_\Gamma h) z \overline \mbh_\eps' + (O_{C^{1,\alpha}_\eps}(\nabla_\Gamma h))^2 \overline \mbh_\eps'' + O_{1,0,\alpha,\eps}(\eps^3).
\end{align}

\begin{rema} \label{rema:mult.one.order}
	We split up the right hand side of \eqref{eq:mult.one.phi.pde} into three lines according to the order of contribution of each term once sharp estimates have been derived. The first line is $O(\eps)$, the second is $O(\eps^2)$, and the third is $o(\eps^2)$.
\end{rema}

Following \cite[(10.2)]{WangWei}, we project \eqref{eq:mult.one.phi.pde} onto $\Gamma$ by dotting with $\overline \mbh_\eps'$ along the $z$ coordinate (see Appendix \ref{app:mult.one.h.pde}). We get that, on $B_{19}^\Gamma$:
\begin{align} \label{eq:mult.one.h.pde}
	& 2 \mfh(\cdot, 0) - (\overline \energyunit + \langle \phi, \overline \mbh_\eps'' \rangle_{L^2(\RR)}) (H_\Gamma - \Delta_\Gamma h)  \nonumber \\
	& \qquad = - 2 \big[ (\partial_z \mfh(\cdot, 0)) + (|\sff_\Gamma|^2 + \ricc(\bm{n}_\Gamma, \bm{n}_\Gamma)) \big] h \nonumber \\
	& \qquad + O_{1,0,\alpha,\eps}(\eps^2 \nabla_\Gamma^2 \phi, \eps \nabla_\Gamma \phi) + \eps^{-1} (O_{1,0,\alpha,\eps}(\eps \nabla_\Gamma \phi))^2 + \eps^{-1} (O_{1,0,\alpha,\eps}(\phi))^3 \nonumber \\
	& \qquad + \eps \cdot O_{1,0,\alpha,\eps}(\phi) + \eps^{-1} (O_{1,0,\alpha,\eps}(\phi))^2 + O_{1,0,\alpha,\eps}(\eps^2).
\end{align}
We will later refine the $\eps^{-1} (O_{1,0,\alpha,\eps}(\phi))^2$ term. By \eqref{eq:mult.one.h.pde} we have, on $B_{19}^\Gamma$:
\begin{multline} \label{eq:mult.one.h.pde.2}
	\eps (\overline \energyunit + \langle \phi, \overline \mbh_\eps'' \rangle_{L^2(\RR)}) (H_\Gamma - \Delta_\Gamma h) = 2 \eps \mfh(\cdot, 0) \\
	+ O_{1,0,\alpha,\eps}(\eps^2) + (O_{1,0,\alpha,\eps}(\eps^2 \nabla^2_\Gamma \phi, \eps \nabla_\Gamma \phi, \phi))^2,
\end{multline}

This form of \eqref{eq:mult.one.phi.pde}, \eqref{eq:mult.one.h.pde.2} is convenient (and powerfully exploited in \cite{WangWei}) in that one can use the stability of the one-dimensional model operator $\tfrac{d^2}{dt^2} - W''(\mbh)$ to estimate $\phi$ in terms of the right hand side of $\eps^2 \Delta \phi - W''(\overline \mbh_\eps)$, while at the same time using \eqref{eq:mult.one.h.pde.2} to estimate the term $\eps (H_\Gamma - \Delta_\Gamma h)\overline \mbh_\eps'$ that appears in the right hand side of \eqref{eq:mult.one.phi.pde}; see Appendix \ref{app:estimate.orthogonal.to.h} for an exposition in this multiplicity-one setting. By an iteration scheme we find that:
\begin{equation} \label{eq:mult.one.h.phi.estimates}
	\eps^{-1} \Vert h \Vert_{C^{2,\alpha}_\eps(B_{18}^\Gamma)} + \Vert \phi \Vert_{C^{2,\alpha}_\eps(B_{18}^\Gamma \times (-1, 1))} = O(\eps).
\end{equation}
We may in turn plug this estimate into \eqref{eq:mult.one.h.pde} to also find that:
\begin{equation} \label{eq:mult.one.mean.curv.estimates}
	\Vert H_\Gamma - \Delta_\Gamma h - 2 \energyunit^{-1} \mfh(\cdot, 0) \Vert_{C^{0,\alpha}_\eps(B^\Gamma_{18})} = O(\eps).
\end{equation}
While the $\eps \mfh$ term of \eqref{eq:ac.pde.h} curtails the estimate one can get on $\phi$ in \eqref{eq:mult.one.h.phi.estimates} (cf. \cite[Section 15]{WangWei}), one does still get the improved estimate on horizontal derivatives of $\phi$ as in \cite[Section 13]{WangWei}. The point is that, when we take the tangential derivative of \eqref{eq:mult.one.phi.pde}, the effect of the tangential derivative of the term that was previously the bottleneck, $\eps \mfh$, does not scale like $O(\eps^{-1})$ as all the other terms do. Thus, as in \cite[Section 13]{WangWei}  (see Appendix \ref{app:estimate.orthogonal.to.h}) one has:
\begin{equation} \label{eq:mult.one.h.phi.horizontal.estimates}
	\Vert \nabla_\Gamma h \Vert_{C^{2,\alpha}_\eps(B_{17}^\Gamma)} + \Vert \eps \nabla_\Gamma \phi \Vert_{C^{2,\alpha}_\eps(B_{17}^\Gamma \times (-1, 1))} = O(\eps^2).
\end{equation}
This in turn lets us refine \eqref{eq:mult.one.mean.curv.estimates} to (cf. \cite[Section 15]{WangWei})\footnote{\label{foot:c3alpha} We note that \cite[Section 15]{WangWei} only states the $C^{0,\alpha}_\eps$ estimates. Higher order estimates were derived in \cite[Section 13]{WangWei} in the form of $W^{1,p}_\eps$ estimates, and were allured to in \cite[Section 7]{WangWei2} in the form of $C^{1,\alpha}_\eps$ estimates.}
\begin{equation} \label{eq:mult.one.mean.curv.estimates.refined}
	H_\Gamma - 2 \energyunit^{-1} \mfh(\cdot, 0) = O_{C^{1,\alpha}_\eps}(\eps).
\end{equation}

\begin{rema} \label{rema:mult.one.mean.curv.estimates.tau}
	One can similarly estimate the mean curvature of $\{ u = \tau \}$ for $\tau < \tfrac12 \beta_0$ by working with $h + \eps \tau$ in place of $h$ in Fermi coordinates off $\Gamma_\tau := \{ u = \tau \}$.
\end{rema}

Now, in order to get the full $\eps$-term in \eqref{eq:u.expansion}, we adapt (and simplify) the ansatz of \cite{AlikakosFuscoStefanopoulos96}\footnote{A key difference with \cite{AlikakosFuscoStefanopoulos96} is that we are trying to understand an arbitrary solution, not a particular solution with tailored asymptotics.} and consider an auxiliary correction function: the unique bounded solution $\mbi : \RR \to \RR$ of 
\begin{equation} \label{eq:mult.one.i.ode}
	\mbi''(t) - W''(\mbh(t)) \mbi(t) = 1 - 2 \energyunit^{-1} \mbh'(t), \; \mbi(0) = 0.
\end{equation}
This $\mbi$ converges exponentially to $\mbi(\pm \infty) = -1/W''(\pm 1)$ as $|t| \to \infty$. For the existence and exponential asymptotics of $\mbi$ we refer the reader to \cite[Lemma B.1, Remark B.3]{AlikakosFuscoStefanopoulos96}. Having an exponential tail, as $\mbh$ does, $\mbi$ also satisfies:
\begin{equation} \label{eq:mbi.times.polynomial}
	\sup_{z \in \RR} |z|^k |\mbi^{(\ell)}(z)| < \infty \text{ for all } k, \ell \in \NN, \; \ell \geq 1,
\end{equation}
and, moreover, cutting off $\mbi$ as we did $\mbh$ in \eqref{eq:mbh.cutoff}, we denote:
\begin{equation} \label{eq:mbi.cutoff}
	\overline \mbi(t) := \chi(\eps^{\delta_*} t) \mbi(t) + \mbi(\pm \infty) (1 - \chi(\eps^{\delta_*} t)),
\end{equation}
so that
\begin{equation} \label{eq:mbi.cutoff.estimates}
	|\overline \mbi'' - W''(\overline \mbh) \overline \mbi - 1 + 2 \energyunit^{-1} \overline \mbh'(t)|_{C^3(\RR)} = O(\eps^3).
\end{equation}
We similarly denote:
\begin{align*}
	\overline \mbi_\eps(y,z) & := \overline \mbi(\eps^{-1}(z-h(y))), \\
	\overline \mbi_\eps'(y, z) & ;= \overline \mbi'(\eps^{-1}(z-h(y))), \\
	\overline \mbi_\eps''(y, z) & := \overline \mbi''(\eps^{-1}(z-h(y))).
\end{align*}
As before, we compute, using \eqref{eq:mult.one.sff.bounds}, \eqref{eq:mult.one.ddt.connection}, \eqref{eq:mbh.cutoff.estimates},  \eqref{eq:mbh.times.polynomial}, \eqref{eq:mbi.times.polynomial}, \eqref{eq:mbi.cutoff.estimates}:
\begin{align}
	\eps^2 \Delta (\eps \mfh \overline \mbi_\eps) 
		& = \eps^3 \mfh \Delta \overline \mbi_\eps + 2 \eps^3 \nabla \mfh \cdot \nabla \overline \mbi_\eps + \eps^3 (\Delta \mfh) \overline \mbi_\eps \nonumber \\
		& = \eps^3 \mfh (\Delta_{g_z} + H_z \partial_z + \partial_z^2) \overline \mbi_\eps + 2 \eps^2 (\partial_z \mfh - \nabla_{g_z} \mfh \cdot \nabla_{g_z} h) \overline \mbi'_\eps + O_{1,0,\alpha,\eps}(\eps^3) \nonumber \\
		& = \eps \mfh(|\nabla_{g_z} h|^2 \overline \mbi_\eps'' - \eps (\Delta_{g_z} h) \overline \mbi_\eps') + \eps^2 \mfh H_z \overline \mbi_\eps' + \eps \mfh \overline \mbi_\eps'' + 2 \eps^2 (\partial_z \mfh) \overline \mbi_\eps' + O_{1,0,\alpha,\eps}(\eps^3) \nonumber \\
		& = \eps \mfh (W''(\overline \mbh_\eps) \overline \mbi_\eps + 1 - 2 \energyunit^{-1} \overline \mbh_\eps') + \eps^2 \mfh H_z \overline \mbi_\eps' + 2 \eps^2 (\partial_z \mfh) \overline \mbi_\eps' + O_{1,0,\alpha,\eps}(\eps^3) \nonumber \\
		& = W''(\overline \mbh_\eps) \eps \mfh \overline \mbi_\eps + \eps \mfh - 2 \eps \mfh \energyunit^{-1} \overline \mbh_\eps' + \eps^2 \mfh H_z \overline \mbi_\eps' + 2 \eps^2 (\partial_z \mfh) \overline \mbi_\eps' + O_{1,0,\alpha,\eps}(\eps^3) \nonumber \\
		& = W''(\overline \mbh_\eps) \eps \mfh \overline \mbi_\eps + \eps \mfh - 2 \eps (\mfh(\cdot, 0) + (\partial_z \mfh)(\cdot, 0) z + O_{1,0,\alpha,\eps}(1) z^2) \energyunit^{-1} \overline \mbh_\eps' \nonumber \\
		& \qquad + \eps^2 (\mfh(\cdot, 0) + O_{1,0,\alpha,\eps}(1)z) (H_\Gamma + O_{1,0,\alpha,\eps}(1) z) \overline \mbi_\eps' \nonumber \\
		& \qquad + 2 \eps^2 ((\partial_z \mfh)(\cdot, 0) + O_{1,0,\alpha,\eps}(1) z) \overline \mbi_\eps' + O_{1,0,\alpha,\eps}(\eps^3) \nonumber \\
		& = W''(\overline \mbh_\eps) \eps \mfh \overline \mbi_\eps + \eps \mfh - 2 \eps \energyunit^{-1} \mfh(\cdot, 0) \overline \mbh_\eps' \nonumber \\
		& \qquad - 2 \eps \energyunit^{-1} (\partial_z \mfh)(\cdot, 0) z \overline \mbh_\eps' + \eps^2 \mfh(\cdot, 0) H_\Gamma \overline \mbi_\eps' + 2 \eps^2 (\partial_z \mfh)(\cdot, 0) \overline \mbi_\eps' + O_{1,0,\alpha,\eps}(\eps^3). \nonumber
\end{align}
i.e.,
\begin{align}  \label{eq:mult.one.mbi.pde}
	& \eps^2 \Delta (\eps \mfh \overline \mbi_\eps) - W''(\overline \mbh_\eps) \eps \mfh \overline \mbi_\eps \nonumber \\
	& \qquad = \eps \mfh - 2 \eps \energyunit^{-1} \mfh(\cdot, 0) \overline \mbh_\eps' \nonumber \\
	& \qquad - 2 \eps \energyunit^{-1} (\partial_z \mfh)(\cdot, 0) z \overline \mbh_\eps' + \eps^2 \big[ \mfh(\cdot, 0) H_\Gamma + 2 (\partial_z \mfh)(\cdot, 0) \big] \overline \mbi_\eps' + O_{1,0,\alpha,\eps}(\eps^3).
\end{align}
Plugging \eqref{eq:mult.one.mbi.pde} into \eqref{eq:mult.one.phi.pde} gives an equation for
\[ \hat \phi := \phi - \eps \mfh \overline \mbi_\eps \, ( = u - \overline \mbh_\eps - \eps \mfh \overline \mbi_\eps), \]
which is:
\begin{align} \label{eq:mult.one.phihat.pde}
	& \eps^2 \Delta \hat \phi - W''(\overline \mbh_\eps) \hat \phi - \tfrac12 W'''(\overline \mbh_\eps) \hat \phi(\hat \phi + 2 \eps \mfh \overline \mbi_\eps) \nonumber \\
	& \qquad = \eps \big[ 2 \energyunit^{-1} \mfh(\cdot, 0) - (H_\Gamma - \Delta_\Gamma h) \big] \overline \mbh_\eps'  \nonumber \\
	& \qquad + \eps \big[ (|\sff_\Gamma|^2 + \ricc(\bm{n}_\Gamma, \bm{n}_\Gamma)) + 2 \energyunit^{-1} (\partial_z \mfh)(\cdot, 0) \big] z \overline \mbh_\eps' \nonumber \\
	& \qquad - \eps^2 \big[ \mfh(\cdot, 0) H_\Gamma + 2 (\partial_z \mfh)(\cdot, 0) \big] \overline \mbi_\eps' \nonumber \\
	& \qquad + \tfrac12 \eps^2 \mfh^2 W'''(\overline \mbh_\eps) \overline \mbi_\eps^2 + O_{1,0,\alpha,\eps}(\eps^3). 
\end{align}
Notice that all terms on the right hand side are $O_{C^{0,\alpha}_\eps}(\eps^2)$, while the extra term on the left hand side that is not part of the stability operator is $o_{C^{2,\alpha}_\eps}(1) \hat \phi$. As before (see Appendix \ref{app:estimate.orthogonal.to.h}), 
\begin{equation} \label{eq:mult.one.phihat.estimate}
	\Vert \hat \phi \Vert_{C^{2,\alpha}_\eps(B^\Gamma_{16} \times (-1,1))} = O(\eps^2).
\end{equation}
This in turn lets us further refine \eqref{eq:mult.one.mean.curv.estimates.refined} (see Appendix \ref{app:mult.one.h.pde}) to (cf. \cite[Lemma 5.5]{ChodoshMantoulidis:multiplicity-one})
\begin{equation} \label{eq:mult.one.mean.curv.estimates.more.refined}
	\Vert H_\Gamma - \Delta_\Gamma h - 2 \energyunit^{-1} \mfh(\cdot, 0) \Vert_{C^{0,\alpha}_{\eps}(B^\Gamma_{16})} = O(\eps^2).
\end{equation}

Finally, we compute the $\eps^2$-order terms in \eqref{eq:u.expansion}. To do so, we consider the unique bounded ODE solutions of
\begin{equation} \label{eq:mult.one.j.ode}
	\mbj''(t) - W''(\mbh(t)) \mbj(t) = t \mbh'(t), \; \mbj(0) = 0,
\end{equation}
\begin{equation} \label{eq:mult.one.k.ode}
	\mbk''(t) - W''(\mbh(t)) \mbk(t) = \mbi'(t), \; \mbk(0) = 0,
\end{equation}
\begin{equation} \label{eq:mult.one.l.ode}
	\mbl''(t) - W''(\mbh(t)) \mbl(t) = W'''(\mbh(t)) \mbi(t)^2, \; \mbl(0) = 0.
\end{equation}
Again, we defer to \cite[Lemma B.1, Remark B.3]{AlikakosFuscoStefanopoulos96} for the existence and exponential asymptotics of $\mbj$, $\mbk$, $\mbl$. Similarly to before, denote
\[ \mbj_\eps(y, z) := \mbj(\eps^{-1}(z-h(y))), \; \mbk_\eps(y, z) := \mbk(\eps^{-1}(z-h(y))), \mbl_\eps(y, z) := \mbl(\eps^{-1}(z-h(y))). \]
(We do not need to truncate these ODE solutions.) Denote:
\begin{align*}
	\tilde \phi 
		& := \hat \phi - \eps^2 \big[ (|\sff_\Sigma|^2 + \ricc(\bm{n}_\Sigma, \bm{n}_\Sigma)) \circ \Pi_\Sigma + 2 \energyunit^{-1} (\partial_z \mfh)(\cdot, 0) \big] \mbj_\eps \\
		& \qquad + \eps^2 \big[ \mfh(\cdot, 0) (H_\Sigma \circ \Pi_\Sigma) + 2 (\partial_z \mfh)(\cdot, 0) \big] \mbk_\eps - \tfrac12 \eps^2 \mfh^2 \mbl_\eps,
\end{align*} 
where $\Sigma$ is the $C^{2,\alpha}$ limit of $\Gamma$ as $\eps \to 0$, which has $H_\Sigma = 2 \energyunit^{-1} \mfh|_\Sigma$, and $\Pi_\Sigma$ is the projection onto $\Sigma$. Working as we did to get to \eqref{eq:mult.one.mbi.pde}, and using \eqref{eq:mult.one.h.phi.estimates}, \eqref{eq:mult.one.h.phi.horizontal.estimates}, \eqref{eq:mult.one.mean.curv.estimates.more.refined}, \eqref{eq:mult.one.j.ode}, \eqref{eq:mult.one.k.ode}, \eqref{eq:mult.one.l.ode}, we find that
\begin{equation} \label{eq:mult.one.phitilde.pde}
	\eps^2 \Delta \tilde \phi - W''(\overline \mbh_\eps) \tilde \phi = o(\eps^2),
\end{equation}
near $\Gamma$ so, arguing as in \cite[Proposition 5.6]{ChodoshMantoulidis:multiplicity-one}, we find that, near $\Gamma$:
\begin{equation} \label{eq:mult.one.phitilde.estimate}
	\tilde \phi = o(\eps^2).
\end{equation}

\section{Lower bounds for eigenvalues of $\delta^2 E_{\eps,\mfh}$ as $\eps \to 0$}

\label{sec:index.lower}

\begin{theo} \label{theo:index.lower}
	Let $(M^n, g)$ be a closed Riemannian manifold. Consider a sequence of critical points $u_i$ to $E_{\eps_i, \mfh_i}$ with $\eps_i \to 0$, $\Vert \mfh_i \Vert_{C^{3,\alpha}(M)} + E_{\eps_i,\mfh_i}[u_i] \leq E$ for all $i$, and a fixed $\alpha \in (0, 1)$. Let $V$ denote the limiting varifold and $\Omega$ denote the limiting domain of $(u_i, \eps_i)$, and $\mfh$ denote the limiting $\mfh_i$ after passing to a subsequence $\eps_{i'} \to 0$. Assume $V$ is a multiplicity-one varifold associated to a smooth hypersurface $\Sigma \subset M$. Fix $\lambda_0 \in \RR$. For sufficiently small $\eps > 0$,
	\begin{multline*}
		\# \{ \text{eigenvalues } \lambda \leq \lambda_0 \text{ (with multiplicity) of } \delta^2 A_\mfh[V; \Omega] \} \\
		\geq \# \{ \text{eigenvalues } \lambda \leq \eps \lambda_0 \text{ (with multiplicity) of } \delta^2 E_{\eps,\mfh}[u] \}.
	\end{multline*}
\end{theo}

Note that Theorem \ref{theo:index.lower} together with \eqref{eq:index.upper} also implies:

\begin{coro} \label{coro:index.lower}
	For any $\ell \in \NN$, let $\lambda_\ell(A_\mfh[\Sigma; \Omega])$ and $\lambda_\ell(E_{\eps,\mfh}[u])$ denote the $\ell$-th eigenvalues of $\delta^2 A_\mfh[\Sigma; \Omega]$ and $\delta^2 E_{\eps,\mfh}[u]$, respectively. Then
	\begin{equation} \label{eq:index.lower}
		\lambda_\ell(A_\mfh[\Sigma; \Omega]) = \lim_{\eps \to 0} \eps^{-1} \lambda_\ell(E_{\eps,\mfh}[u]).
	\end{equation}
\end{coro}

\begin{rema} \label{rema:index.lower.known}
	Theorem \ref{theo:index.lower}, Corollary \ref{coro:index.lower}, and their consequence in Theorem \ref{theo:higher.dim} (c), bound from below the index-plus-nullity of a smooth, multiplicity-one limiting $(V; \Omega)$ in terms of the indices-plus-nullities of $(u_i, \eps_i, \mfh_i)$. This generalizes what was known for $\mfh_i \equiv 0$ from \cite[Theorem 5.11]{ChodoshMantoulidis:multiplicity-one} (cf. \cite{CajuGaspar:symmetry}) to the setting of arbitrary $\mfh_i$.
\end{rema}

Given the sharp approximations in Section \ref{sec:mult.one}, the argument for Theorem \ref{theo:index.lower} can be modeled after \cite[Theorem 5.11]{ChodoshMantoulidis:multiplicity-one}. Aspects of the original proof in \cite{ChodoshMantoulidis:multiplicity-one} have been simplified. We also encounter certain other difficulties coming from the $\mfh$ term that we resolve. 

As before, we denote $\Gamma := \{ u = 0 \}$; it is a smooth closed hypersurface, and converges to $\Sigma$ in $C^{2,\alpha}$ as $\eps \to 0$ by \eqref{eq:mult.one.h.phi.estimates}, \eqref{eq:mult.one.mean.curv.estimates}. We introduce the notation:
\[ \cQ_{u,\eps}(\zeta, \xi) = \int_{M} \big[ \eps \nabla \zeta \cdot \nabla \xi + \eps^{-1} W''(u) \zeta \xi \big] \, d\mu_g, \; \zeta, \xi \in C^\infty(M), \]
\[ \cQ_\Gamma(\zeta, \xi) = \int_\Gamma \big[ (\nabla_\Gamma \zeta \cdot \nabla_\Gamma \xi - (|\sff_\Gamma|^2 + \ricc(\bm{n}_\Gamma, \bm{n}_\Gamma) + 2 \energyunit^{-1} (\partial_z \mfh)(\cdot, 0)) \zeta \xi \big] \, d\mu_\Gamma, \; \zeta, \xi \in C^\infty(\Gamma). \]
These quadratic forms relate to the second variations of $E_{\eps,\mfh}[u]$ and $A_{2 \energyunit^{-1} \mfh}[\Gamma]$ in \eqref{eq:second.variation.ac} and \eqref{eq:second.variation.h.smooth.scalar}---though, keep in mind that $\Gamma$ is not a critical point of $A_\mfh[\cdot])$, so $\cQ_\Gamma$ isn't its second variation. It will be convenient to work locally within a fixed $\eta > 0$ tubular neighborhood $\cU \subset M$ of $\Gamma$, and further denote:
\[ \cQ_{u,\eps}^\cU(\zeta, \xi) =  \int_{\cU} \big[ \eps \nabla \zeta \cdot \nabla \xi + \eps^{-1} W''(u) \zeta \xi \big] \, d\mu_g, \; \zeta, \xi \in C^\infty(\cU). \]

Consider an arbitrary $w \in C^\infty(\cU)$. Working in Fermi coordinates $(y, z) \in \Gamma \times (-\eta, \eta) = \cU$ over $\Gamma$, we may decompose $w$ as:
\begin{equation} \label{eq:index.lower.w.decomposition}
	w(y, z) =: w^\parallel(y) (\overline \mbh_\eps'(y,z) + \eps \mfh \overline \mbi_\eps'(y,z)) + w^\perp(y, z),
\end{equation}
where
\begin{equation} \label{eq:index.lower.wperp}
	\int_{-\eta}^\eta w^\perp(y, z) (\overline \mbh_\eps'(y, z) + \eps \mfh \overline \mbi_\eps'(y,z)) \, dz = 0 \text{ for every } y \in \Gamma.
\end{equation}
As in \cite[Section 5]{ChodoshMantoulidis:multiplicity-one}, one has:
\begin{equation} \label{eq:index.lower.w.l2}
	\int_\cU w^2 \, d\mu_g = \eps(\energyunit + o(1)) \int_\Gamma (w^\parallel)^2 \, d\mu_\Gamma + (1+o(1)) \int_\cU (w^\perp)^2 \, d\mu_g.
\end{equation}
We will need the following important lemmas regarding the behavior of $\cQ_{u,\eps}^\cU$ with respect to this decomposition. Their proofs are given at the end of the section.
\begin{lemm} \label{lemm:index.lower.wpar}
	For $\eps > 0$ sufficiently small, and all $f$, $g \in C^\infty(\Gamma)$:
	\[ \cQ_{u,\eps}^\cU(f (\overline \mbh_\eps' + \eps \mfh \overline \mbi_\eps'), g (\overline \mbh_\eps' + \eps \mfh \overline \mbi_\eps')) = \eps^2 \energyunit \cQ_\Gamma(f, g) + o(\eps^2) \int_\Gamma \big[ |\nabla_\Gamma f|^2 + |\nabla_\Gamma g|^2 + f^2 + g^2 \big] \, d\mu_\Gamma. \]
\end{lemm}
\begin{lemm} \label{lemm:index.lower.cross}
	For $\eps > 0$ sufficiently small, all $f \in C^\infty(\Gamma)$, and all $\omega \in C^\infty(\cU)$ satisfying \eqref{eq:index.lower.wperp}:
	\[ \cQ_{u,\eps}^\cU(f (\overline \mbh_\eps' + \eps \mfh \overline \mbi_\eps'), \omega) = o(\eps^2) \int_\Gamma \big[ |\nabla_\Gamma f|^2 + f^2 \big] \, d\mu_\Gamma + o(1) \int_\cU \big[ \eps |\nabla \omega|^2 + \eps^{-1} \omega^2 \big] \, d\mu_g. \]
\end{lemm}
\begin{lemm} \label{lemm:index.lower.wperp}
	There is a constant $\gamma_1 > 0$ so that for $\eps > 0$ sufficiently small and all $\omega \in C^\infty(\cU)$ satisfying \eqref{eq:index.lower.wperp}:
	\[ \cQ_{u,\eps}^\cU(\omega, \omega) \geq \gamma_1 \int_\cU \big[ \eps |\nabla \omega|^2 + \eps^{-1} \omega^2 \big] \, d\mu_g. \]
\end{lemm}

These lemmas have a few straightforward but important implications. Together with \eqref{eq:index.lower.w.l2}, they show that for all $\eps > 0$ sufficiently small and all $w \in C^\infty(\cU)$,
\begin{equation} \label{eq:index.lower.Q.general.lower}
	\cQ_{u,\eps}^\cU(w, w) \geq - \eps \gamma_2 \int_{\cU} w^2 \, d\mu_g
\end{equation}
for some fixed $\gamma_2 > 0$. (See \cite[Lemma 5.10]{ChodoshMantoulidis:multiplicity-one}.) Using also that $W''(u) \geq \kappa > 0$ on $M \setminus \cU$ for $\eps > 0$ small, we note that \eqref{eq:index.lower.Q.general.lower} implies that
\[ \cQ_{u,\eps}(w, w) \geq \cQ_{u,\eps}^{\cU}(w, w) + \int_{M \setminus \cU} \eps^{-1} W''(u) w^2 \, d\mu_g \geq - \eps \gamma_2 \int_{\cU} w^2 \, d\mu_g + \eps^{-1} \kappa \int_{M \setminus \cU} w^2 \, d\mu_g \]
and thus we get the strong $L^2$ localization estimate
\begin{equation} \label{eq:index.lower.localization}
	\int_{M \setminus \cU} w^2 \, d\mu_g \leq C(\Lambda) \eps^2 \int_\cU w^2 \, d\mu_g, \; \text{ provided }   \cQ_{u,\eps}(w, w) \leq \eps \Lambda \int_M \omega^2 \, d\mu_g.
\end{equation}
Let us now show how Theorem \ref{theo:index.lower} follows from these facts.

\begin{proof}[Proof of Theorem \ref{theo:index.lower}]
	Denote
	\[ I_\Sigma := \# \{ \text{eigenvalues } \lambda \leq \lambda_0 \text{ (with multiplicity) of } \delta^2 A_\mfh[V; \Omega] \}, \]
	\[ I_\eps := \# \{ \text{eigenvalues } \lambda \leq \eps \lambda_0 \text{ (with multiplicity) of } \delta^2 E_{\eps,\mfh}[u] \}. \]
	From the variational characterization of eigenvalues of $\delta^2 A_\mfh[V; \Omega]$, the discrete nature of the corresponding spectrum, and the $C^2$ convergence of $\Gamma = \{ u = 0 \}$ to $\Sigma$, there exist $\gamma_3 > 0$ and functions $f_{1}, \ldots, f_{I_\Sigma} : \Gamma \to \RR$ such that
	\begin{equation} \label{eq:index.lower.positive.gap.gamma}
		\cQ_\Gamma(f, f) \geq \lambda_0 \int_\Gamma f^2 \, d\mu_\Gamma + \gamma_3 \int_\Gamma \big[ |\nabla_\Gamma f|^2 + f^2 \big] \, d\mu_\Gamma
	\end{equation}
	for all $f \in C^1(\Gamma)$ satisfying $\langle f, f_i \rangle_{L^2(\Gamma)} = 0$ for every $i = 1, \ldots, I_{\Sigma}$. Consider the linear map $\cI_\Gamma : L^2(\Gamma) \to \RR^{I_\Sigma}$ given by
	\[ \cI_\Gamma(f) := (\langle f, f_1 \rangle_{L^2(\Gamma)}, \ldots, \langle f_, f_{I_\Sigma} \rangle_{L^2(\Gamma)}). \]
	Using \eqref{eq:index.lower.w.l2}, \eqref{eq:index.lower.positive.gap.gamma}, and Lemmas \ref{lemm:index.lower.wpar}, \ref{lemm:index.lower.cross}, \ref{lemm:index.lower.wperp}, we find that for all $w \in C^\infty(\cU)$ with $w^\parallel \in \ker \cI_\Gamma$,
	\begin{align} \label{eq:index.lower.positive.gap.u}
		& \cQ_{u,\eps}^\cU(w, w) \nonumber \\
		& \qquad = \cQ_{u,\eps}^\cU(w^\parallel (\overline \mbh_\eps' + \eps \mfh \overline \mbi_\eps'), w^\parallel (\overline \mbh_\eps' + \eps \mfh \overline \mbi_\eps')) + \cQ_{u,\eps}^\cU(w^\perp, w^\perp) + 2 \cQ_{u,\eps}^\cU(w^\parallel (\overline \mbh_\eps' + \eps \mfh \overline \mbi_\eps'), w^\perp) \nonumber \\
		& \qquad \geq \eps^2(\energyunit - o(1)) \cQ_\Gamma(w^\parallel, w^\parallel) + o(\eps^2) \int_\Gamma \big[ |\nabla_\Gamma w^\parallel|^2 + (w^\parallel)^2 \big] \, d\mu_\Gamma \nonumber \\
		& \qquad \qquad + \gamma_1 \int_\cU \big[ \eps |\nabla w^\perp|^2 + \eps^{-1} (w^\perp)^2 \big] \, d\mu_g \nonumber \\
		& \qquad \qquad + o(\eps^2) \int_\Gamma \big[ |\nabla_\Gamma w^\parallel|^2 + (w^\parallel)^2 \big] \, d\mu_\Gamma + o(1) \int_\cU \big[ \eps |\nabla w^\perp|^2 + \eps^{-1} (w^\perp)^2 \big] \, d\mu_g \nonumber \\
		& \qquad \geq \eps^2 \energyunit \cdot \lambda_0 \int_\Gamma (w^\parallel)^2 \, d\mu_\Gamma + \eps^2 \energyunit \cdot \tfrac12 \gamma_3 \int_\Gamma \big[ |\nabla_\Gamma w^\parallel|^2 + (w^\parallel)^2 \big] \, d\mu_\Gamma \nonumber \\
		& \qquad \qquad + \tfrac12 \gamma_1 \int_\cU \big[ \eps |\nabla w^\perp|^2 + \eps^{-1} (w^\perp)^2 \big] \, d\mu_g \nonumber \\
		& \qquad \geq \eps (\lambda_0 + \gamma_4) \int_\cU w^2 \, d\mu_g,
	\end{align}
	for some $\gamma_4 > 0$. 
	
	We now prove that $I_\Sigma \geq I_\eps$. Let $w_1, \ldots, w_{I_\eps} \in C^\infty(M)$ denote an $L^2(M)$-orthonormal set of eigenfunctions of $\cQ_{u,\eps}$ with eigenvalues $\leq \lambda_0 \eps$, and set:
	\[ W_\Gamma := \lspan \{ w_1^\parallel, \ldots, w_{I_\eps}^\parallel \} \subset C^\infty(\Gamma). \]
	If $I_\Sigma \geq I_\eps$ failed, there would exist $w \in \lspan\{ w_1, \ldots, w_{I_\eps} \} \setminus \{0\}$ with $w^\parallel \in \ker \cI_\Gamma$. By \eqref{eq:index.lower.positive.gap.u},
	\[ \eps \lambda_0 \int_M \omega^2 \, d\mu_g \geq \cQ_{u,\eps}(w, w) \geq \cQ_{u,\eps}^\cU(w, w) \geq \eps (\lambda_0 + \gamma_3) \int_\cU w^2 \, d\mu_g, \]
	which implies that $w \equiv 0$ on $M$ by \eqref{eq:index.lower.localization}, a contradiction.
\end{proof}

\begin{proof}[Proof of Lemma \ref{lemm:index.lower.wpar}]
	We have:
	\begin{align} \label{eq:index.lower.wpar.full}
		& \cQ_{u,\eps}^\cU(f (\overline \mbh_\eps' + \eps \mfh \overline \mbi_\eps'), g (\overline \mbh_\eps' + \eps \mfh \overline \mbi_\eps')) \nonumber \\
		& \qquad = \int_{-\eta}^\eta \int_{\Gamma} \Big[ -\eps f (\overline \mbh_\eps' + \eps \mfh \overline \mbi_\eps') \cdot (\Delta_{g_z} + H_z \partial_z + \partial_z^2)(g (\overline \mbh_\eps' + \eps \mfh \overline \mbi_\eps')) \nonumber \\
		& \qquad \qquad \qquad + \eps^{-1}  W''(u) fg (\overline \mbh_\eps' + \eps \mfh \overline \mbi_\eps)^2 \Big] \, d\mu_{g_z} \, dz \nonumber \\
		& \qquad = \int_{-\eta}^\eta \int_\Gamma \Big[ \eps \nabla_{g_z} (f (\overline \mbh_\eps' + \eps \mfh \overline \mbi_\eps')) \cdot \nabla_{g_z} (g (\overline \mbh_\eps' + \eps \mfh \overline \mbi_\eps')) \nonumber \\
		& \qquad \qquad \qquad - H_z fg (\overline \mbh_\eps'' + \eps \mfh \overline \mbi_\eps'' + \eps^2 (\partial_z \mfh) \overline \mbi_\eps') (\overline \mbh_\eps' + \eps \mfh \overline \mbi_\eps') \nonumber \\
		& \qquad \qquad \qquad + \eps^{-1} fg (W''(u) (\overline \mbh_\eps' + \eps \mfh \overline \mbi_\eps') - \overline \mbh_\eps''' - \eps \mfh \overline \mbi_\eps''' - 2 \eps^2 (\partial_z \mfh) \overline \mbi_\eps'' - \eps^3 (\partial_z^2 \mfh) \overline \mbi_\eps') \nonumber \\
		& \qquad \qquad \qquad \qquad \qquad \cdot (\overline \mbh_\eps' + \eps \mfh \overline \mbi_\eps') \Big] \, d\mu_{g_z} \, dz.
	\end{align}
	We compute/estimate the terms in \eqref{eq:index.lower.wpar.full} one by one. We will repeatedly use Cauchy--Schwarz, \eqref{eq:mbh.cutoff}, \eqref{eq:mbh.times.polynomial}, \eqref{eq:mult.one.h.phi.horizontal.estimates}, \eqref{eq:mbi.times.polynomial}, \eqref{eq:mbi.cutoff}, $d\mu_{g_z} = (1+ H_\Gamma z - \tfrac12 (|\sff_\Gamma|^2 + \ricc(\bm{n}_\Gamma, \bm{n}_\Gamma) z^2 + O(1)z^3)d\mu_\Gamma$, and $H_z = H_\Gamma - (|\sff_\Gamma|^2 + \ricc(\bm{n}_\Gamma, \bm{n}_\Gamma))z + O(z^2)$, which follows from the Riccati equation. We thus have:
	\begin{align} \label{eq:index.lower.wpar.full.1}
		& \int_{-\eta}^\eta \int_\Gamma \eps \nabla_{g_z} (f (\overline \mbh_\eps' + \eps \mfh \overline \mbi_\eps')) \cdot \nabla_{g_z} (g (\overline \mbh_\eps' + \eps \mfh \overline \mbi_\eps')) \, d\mu_{g_z} \, dz \nonumber \\
		& \qquad = \int_{-\eta}^\eta \int_\Gamma \eps \big( (\nabla_\Gamma f)(1 + O(1) z)(\overline \mbh_\eps' + \eps \mfh \overline \mbi_\eps') \nonumber \\
		& \qquad \qquad \qquad \qquad \qquad - \eps^{-1} f (\overline \mbh_\eps'' \nabla_{g_z} h + \eps \mfh \overline \mbi_\eps'' \nabla_{g_z} h - \eps^2 (\nabla_{g_z} \mfh) \overline \mbi_\eps') \big) \nonumber \\
		& \qquad \qquad \qquad \cdot \big( (\nabla_\Gamma g)(1 + O(1) z)(\overline \mbh_\eps' + \eps \mfh \overline \mbi_\eps') \nonumber \\
		& \qquad \qquad \qquad \qquad \qquad - \eps^{-1} g (\overline \mbh_\eps'' \nabla_{g_z} h + \eps \mfh \overline \mbi_\eps'' \nabla_{g_z} h - \eps^2 (\nabla_{g_z} \mfh) \overline \mbi_\eps') \big)^2 \, (1 + O(1)z) d\mu_\Gamma \, dz \nonumber \\
		& \qquad = \eps^2 \energyunit \int_\Gamma \nabla_\Gamma f \cdot \nabla_\Gamma g \, d\mu_\Gamma + o(\eps^2) \int_\Gamma \big[ |\nabla_\Gamma f|^2 + |\nabla_\Gamma g|^2 + f^2 + g^2 \big] \, d\mu_\Gamma.
	\end{align}
	Next, using \eqref{eq:mult.one.h.phi.estimates}, and the  integral identities
	\[ \int_\RR \mbh'' \mbh' \, dz = 0, \text{ and } \int_\RR z \mbh'' \mbh' \, dz = - \tfrac12 \int_{\RR} (\mbh')^2 \, dz = - \tfrac12 \energyunit, \]
	and \eqref{eq:mult.one.h.phi.horizontal.estimates}, we have:
	\begin{align} \label{eq:index.lower.wpar.full.2.1}
		& - \int_{-\eta}^\eta \int_\Gamma H_z fg \overline \mbh_\eps'' \overline \mbh_\eps' \, d\mu_{g_z} \, dz \nonumber \\
		& \qquad = - \int_{-\eta}^\eta \int_\Gamma (H_\Gamma - (|\sff_\Gamma|^2 + \ricc(\bm{n}_\Gamma, \bm{n}_\Gamma))z + O(1) z^2) fg \overline \mbh_\eps'' \overline \mbh_\eps' \, (1 + H_\Gamma z + O(1) z^2) \, d\mu_\Gamma \, dz \nonumber \\
		& \qquad = - \int_\Gamma H_\Gamma fg  \Big[ \int_{-\eta}^\eta \overline \mbh_\eps'' \overline \mbh_\eps' \, dz \Big] \, d\mu_\Gamma + \int_\Gamma (|\sff_\Gamma|^2 + \ricc(\bm{n}_\Gamma, \bm{n}_\Gamma) - H_\Gamma^2) fg \Big[ \int_{-\eta}^\eta z \overline \mbh_\eps'' \overline \mbh_\eps' \, dz \Big] \, d\mu_\Gamma \nonumber \\
		& \qquad \qquad + o(\eps^2) \int_\Gamma |fg| \, d\mu_\Gamma \nonumber \\
		& \qquad = - \tfrac12 \energyunit \eps^2  \int_\Gamma (|\sff_\Gamma|^2 + \ricc(\bm{n}_\Gamma, \bm{n}_\Gamma) - H_\Gamma^2) fg \, d\mu_\Gamma + o(\eps^2) \int_\Gamma \big[ f^2 + g^2 \big] \, d\mu_\Gamma.
	\end{align}
	Next, using \eqref{eq:mult.one.h.phi.estimates}, 
	\eqref{eq:mult.one.mean.curv.estimates}, \eqref{eq:mult.one.h.phi.horizontal.estimates}:
	\begin{align} \label{eq:index.lower.wpar.full.2.2}
		& - \int_{-\eta}^\eta \int_\Gamma H_z fg \cdot \eps \mfh \overline \mbi_\eps'' \cdot \overline \mbh_\eps' \, d\mu_{g_z} \, dz \nonumber \\
		& \qquad = - \int_{-\eta}^\eta \int_\Gamma (H_\Gamma + O(1) z) fg \cdot \eps \mfh \overline \mbi_\eps'' \cdot \overline \mbh_\eps' \, (1 + O(1) z) d\mu_\Gamma \, dz \nonumber \\
		& \qquad = \tfrac12 \energyunit \eps^2 \Big[ \int_{\RR} \mbh'' \mbi' \, dz \Big] \int_\Gamma H_\Gamma^2 fg \, d\mu_\Gamma + o(\eps^2) \int_\Gamma \big[ f^2 + g^2 \big] \, d\mu_\Gamma,
	\end{align}
	and
	\begin{align} \label{eq:index.lower.wpar.full.2.3}
		& - \int_{-\eta}^\eta \int_\Gamma H_z fg \cdot \overline \mbh_\eps'' \cdot \eps \mfh \overline \mbi_\eps' \, d\mu_{g_z} \, dz \nonumber \\
		& \qquad = - \int_{-\eta}^\eta \int_\Gamma (H_\Gamma + O(1) z) fg \cdot \overline \mbh_\eps'' \cdot \eps \mfh \overline \mbi_\eps' \, (1 + O(1) z) d\mu_\Gamma \, dz \nonumber \\
		& \qquad = - \tfrac12 \energyunit \eps^2  \Big[ \int_{\RR} \mbh'' \mbi' \, dz \Big] \int_\Gamma H_\Gamma^2 fg \, d\mu_\Gamma + o(\eps^2) \int_\Gamma \big[ f^2 + g^2 \big] \, d\mu_\Gamma,
	\end{align}
	and
	\begin{align} \label{eq:index.lower.wpar.full.2.4}
		& - \int_{-\eta}^\eta \int_\Gamma H_z fg \cdot \big[ \eps \mfh \overline \mbi_\eps'' \cdot \eps \mfh \overline \mbi_\eps' + \eps^2 (\partial_z \mfh) \overline \mbi_\eps' \cdot \overline \mbh_\eps' + \eps^2 (\partial_z \mfh) \overline \mbi_\eps' \cdot \eps \mfh \overline \mbi_\eps' \big] \, d\mu_{g_z} \, dz \nonumber \\
		& \qquad = o(\eps^2) \int_\Gamma fg \, d\mu_\Gamma.
	\end{align}
	Next, using \eqref{eq:mult.one.h.phi.estimates}, \eqref{eq:mult.one.phihat.estimate},   \eqref{eq:mult.one.phitilde.estimate}, the fact that
	\[ \mbi''' - W''(\mbh) \mbi' = W'''(\mbh) \mbh' \mbi - 2 \energyunit^{-1} \mbh'', \]
	and the decomposition $u = \overline \mbh_\eps + \eps \mfh \overline \mbi_\eps + \eps^2 a_\mbj \mbj_\eps + \eps^2 a_\mbk \mbk_\eps + \eps^2 a_\mbl \mbl_\eps + \tilde \phi$, with
	\[ a_{\mbj} = |\sff_\Sigma|^2 + \ricc(\bm{n}_\Sigma, \bm{n}_\Sigma) + 2 \energyunit^{-1} (\partial_z \mfh)(\cdot, 0), \]
	\[ a_{\mbk} = - \mfh(\cdot, 0) H_\Sigma - 2 (\partial_z \mfh)(\cdot, 0), \; a_{\mbl} = \tfrac12 \mfh^2, \]
	we have:
	\begin{align} \label{eq:index.lower.wpar.full.3}
		& \int_{-\eta}^\eta \int_\Gamma \eps^{-1} fg \Big[ W''(u) (\overline \mbh_\eps' + \eps \mfh \overline \mbi_\eps') - \overline \mbh_\eps''' - \eps \mfh \overline \mbi_\eps''' - 2 \eps^2 (\partial_z \mfh) \overline \mbi_\eps'' - \eps^3 (\partial_z^2 \mfh) \overline \mbi_\eps') \Big] (\overline \mbh_\eps' + \eps \mfh \overline \mbi_\eps') \, d\mu_{g_z} \, dz \nonumber \\
		& \qquad = \int_{-\eta}^\eta \int_\Gamma \eps^{-1} fg \Big[ W''(u) (\overline \mbh_\eps' + \eps \mfh \overline \mbi_\eps') - W''(\overline \mbh_\eps)(\overline \mbh_\eps' + \eps \mfh \overline \mbi_\eps') - W'''(\overline \mbh_\eps) \overline \mbh_\eps' \cdot \eps \mfh \overline \mbi_\eps \nonumber \\
		& \qquad \qquad \qquad \qquad \qquad + 2 \eps \energyunit^{-1} \mfh \overline \mbh_\eps'' - 2 \eps^2 (\partial_z \mfh) \overline \mbi_\eps'') \Big] (\overline \mbh_\eps' + \eps \mfh \overline \mbi_\eps') \, d\mu_{g_z} \, dz \nonumber \\
		& \qquad \qquad + o(\eps^2) \int_\Gamma \big[ f^2 + g^2 \big] \, d\mu_\Gamma \nonumber \\
		& \qquad = \eps^{-1} \int_{-\eta}^\eta \int_\Gamma fg \Big[ W''(u) - W''(\overline \mbh_\eps) - W'''(\overline \mbh_\eps) \cdot \eps \mfh \overline \mbi_\eps \Big] (\overline \mbh_\eps' + \eps \mfh_\eps')^2 \, d\mu_{g_z} \, dz \nonumber \\
		& \qquad \qquad + \int_{-\eta}^\eta \int_\Gamma fg \Big[ W'''(\overline \mbh_\eps) \cdot \eps \mfh^2 \overline \mbi_\eps'  \overline \mbi_\eps + 2 \energyunit^{-1} \mfh \overline \mbh_\eps'' - 2 \eps (\partial_z \mfh) \overline \mbi_\eps'' \Big]  (\overline \mbh_\eps' + \eps \mfh \overline \mbi_\eps') \, d\mu_{g_z} \, dz \nonumber \\
		& \qquad \qquad + o(\eps^2) \int_\Gamma \big[ f^2 + g^2 \big] \, d\mu_\Gamma \nonumber \\
		& \qquad = \eps \int_{-\eta}^\eta \int_\Gamma fg W'''(\overline \mbh_\eps) (a_{\mbj} \mbj_\eps + a_{\mbk} \mbk_\eps + a_{\mbl} \mbl_\eps) (\overline \mbh_\eps' + \eps \mfh \overline \mbi_\eps')^2 \, d\mu_{g_z} \, dz \nonumber \\
		& \qquad \qquad + \tfrac12 \eps \int_{-\eta}^\eta \int_{\Gamma} fg W''''(\overline \mbh_\eps) \mfh^2 \overline \mbi_\eps^2 (\overline \mbh_\eps' + \eps \mfh \overline \mbi_\eps')^2 \, d\mu_{g_z} \, dz \nonumber \\
		& \qquad \qquad + \eps \int_{-\eta}^\eta \int_\Gamma fg W'''(\overline \mbh_\eps) \mfh^2 \overline \mbi_\eps'  \overline \mbi_\eps (\overline \mbh_\eps' + \eps \mfh \overline \mbi_\eps') \, d\mu_{g_z} \, dz \nonumber \\
		& \qquad \qquad + 2 \energyunit^{-1} \int_{-\eta}^\eta \int_\Gamma fg \mfh \overline \mbh_\eps'' (\overline \mbh_\eps' + \eps \mfh \overline \mbi_\eps') \, d\mu_{g_z} \, dz \nonumber \\
		& \qquad \qquad - 2 \eps \int_{-\eta}^\eta \int_\Gamma fg (\partial_z \mfh) \overline \mbi_\eps'' (\overline \mbh_\eps' + \eps \mfh \overline \mbi_\eps') \, d\mu_{g_z} \, dz \nonumber \\
		& \qquad \qquad + o(\eps^2) \int_\Gamma \big[ f^2 + g^2 \big] \, d\mu_\Gamma \nonumber \\
		& \qquad = \eps \int_{-\eta}^\eta \int_{\Gamma} fg W'''(\overline \mbh_\eps) (a_{\mbj} \mbj_\eps + a_{\mbk} \mbk_\eps + a_{\mbl} \mbl_\eps) (\overline \mbh_\eps')^2 \, d\mu_{g_z} \, dz \nonumber \\
		& \qquad \qquad + \tfrac12 \eps \int_{-\eta}^\eta \int_{\Gamma} fg W''''(\overline \mbh_\eps) \mfh(\cdot, 0)^2 \overline \mbi_\eps^2 (\overline \mbh_\eps')^2 \, d\mu_{g_z} \, dz \nonumber \\
		& \qquad \qquad + \eps \int_{-\eta}^\eta \int_{\Gamma} fg W'''(\overline \mbh_\eps) \mfh(\cdot, 0)^2 \overline \mbi_\eps' \overline \mbi_\eps \overline \mbh_\eps' \, d\mu_{g_z} \, dz \nonumber \\
		& \qquad \qquad + 2 \energyunit^{-1} \int_{-\eta}^\eta \int_{\Gamma} fg \mfh \overline \mbh_\eps'' (\overline \mbh_\eps' + \eps \mfh \overline \mbi_\eps') \, d\mu_{g_z} \, dz \nonumber \\
		& \qquad \qquad - 2 \eps \int_{-\eta}^\eta \int_{\Gamma} fg (\partial_z \mfh)(\cdot, 0) \overline \mbi_\eps'' \overline \mbh_\eps' \, d\mu_{g_z} \, dz \nonumber \\
		& \qquad \qquad + o(\eps^2) \int_{\Gamma} \big[ f^2 + g^2 \big] \, d\mu_\Gamma.
	\end{align}
	We estimate the terms of \eqref{eq:index.lower.wpar.full.3} individually, leaving the highest order term for last. We repeatedly use the fact that polynomially growing functions $u : \RR \to \RR$ satisfy
	\[ \int_{\RR} W'''(\mbh) \cdot u \cdot (\mbh')^2 \, dz = \int_{\RR} \mbh'' \cdot (u'' - W''(\mbh) u) \, dz, \]
	which is easily checked by integration by parts. Thus, we have
	\[ \int_\RR W'''(\mbh) \cdot \mbj \cdot (\mbh')^2 \, dz = - \tfrac12 \energyunit, \]
	which implies, together with $\Gamma \to \Sigma$ in $C^2$ and $d\mu_{g_z} = (1 + O(1)z) d\mu_\Gamma$, that
	\begin{align} \label{eq:index.lower.wpar.full.3.1}
		& \eps \int_{-\eta}^\eta \int_\Gamma fg W'''(\overline \mbh_\eps) a_{\mbj} \mbj_\eps (\overline \mbh_\eps')^2 \, d\mu_{g_z} \, dz \nonumber \\
		& \qquad = - \tfrac12 \energyunit \eps^2 \int_\Gamma (|\sff_\Gamma|^2 + \ricc(\bm{n}_\Gamma, \bm{n}_\Gamma) + 2 \energyunit^{-1} (\partial_z \mfh)(\cdot, 0)) fg \, d\mu_\Gamma + o(\eps^2) \int_\Gamma \big[ f^2 + g^2 \big] \, d\mu_\Gamma. 
	\end{align}
	Next, we have:
	\[ \int_{\RR} W'''(\mbh) \cdot \mbk \cdot (\mbh')^2 \, dz = \int_\RR \mbh'' \mbi', \]
	which implies, together with \eqref{eq:mult.one.mean.curv.estimates} and $d\mu_{g_z} = (1 + O(1)z) d\mu_\Gamma$, that
	\begin{align} \label{eq:index.lower.wpar.full.3.2}
		& \eps \int_{-\eta}^\eta \int_\Gamma fg W'''(\overline \mbh_\eps) a_{\mbk} \mbk_\eps  (\overline \mbh_\eps')^2 \, d\mu_{g_z} \, dz \nonumber \\
		& \qquad = - \tfrac12 \energyunit \eps^2 \Big[ \int_{\RR} \mbh'' \mbi' \, dz \Big] \int_{\Gamma} (H_\Gamma^2 + 4 \energyunit^{-1} (\partial_z \mfh)(\cdot, 0)) fg \, d\mu_\Gamma + o(\eps^2) \int_{\Gamma} \big[ f^2 + g^2 \big] \, d\mu_\Gamma.
	\end{align}
	Next, we have
	\[ \int_{\RR} W'''(\mbh) \cdot \mbl \cdot (\mbh')^2 \, dz = \int_{\RR} \mbh'' W'''(\mbh) \mbi^2 \, dz, \]
	which implies, together with \eqref{eq:mult.one.mean.curv.estimates}, $d\mu_{g_z} = (1 + O(1)z) d\mu_\Gamma$, and $\Gamma \to \Sigma$ in $C^2$, that
	\begin{align} \label{eq:index.lower.wpar.full.3.3}
		& \eps \int_{-\eta}^\eta \int_\Gamma fg W'''(\overline \mbh_\eps) a_{\mbl} \mbl_\eps (\overline \mbh_\eps')^2 \, d\mu_{g_z} \, dz \nonumber \\
		& \qquad = \tfrac18 \energyunit^2 \eps^2 \Big[ \int_{\RR} W'''(\mbh) \mbi^2 \mbh'' \, dz \Big] \int_{\Gamma} H_\Gamma^2 fg \, d\mu_\Gamma + o(\eps^2) \int_{\Gamma} \big[ f^2 + g^2 \big] \, d\mu_\Gamma.
	\end{align}
	Next, we have by \eqref{eq:mult.one.mean.curv.estimates} and $d\mu_{g_z} = (1 + O(1)z) d\mu_\Gamma$:
	\begin{align} \label{eq:index.lower.wpar.full.3.4}
		& \tfrac12 \eps \int_{-\eta}^\eta \int_\Gamma fg W''''(\overline \mbh_\eps) \mfh(\cdot, 0)^2 \mbi_\eps^2 (\overline \mbh_\eps')^2 \, d\mu_{g_z} \, dz \nonumber \\
		& \qquad = \tfrac18 \energyunit^2 \eps^2 \Big[ \int_{\RR} W''''(\mbh) \mbi^2 (\mbh')^2 \, dz \Big] \int_\Gamma H_\Gamma^2 fg \, d\mu_\Gamma + o(\eps^2) \int_\Gamma \big[ f^2 + g^2 \big] \, d\mu_\Gamma.
	\end{align}
	Next, we have by \eqref{eq:mult.one.mean.curv.estimates} and $d\mu_{g_z} = (1 + O(1)z) d\mu_\Gamma$:
	\begin{align} \label{eq:index.lower.wpar.full.3.5}
		& \eps \int_{-\eta}^\eta \int_{\Gamma} fg W'''(\overline \mbh_\eps) \mfh(\cdot, 0)^2 \overline \mbi_\eps' \overline \mbi_\eps \overline \mbh_\eps' \, d\mu_{g_z} \, dz	\nonumber \\
		& \qquad = \tfrac14 \energyunit^2 \eps^2 \Big[ \int_{\RR} W'''(\mbh) \mbi' \mbi \mbh' \, dz \Big] \int_{\Gamma} H_\Gamma^2 fg \, d\mu_\Gamma + o(\eps^2) \int_\Gamma \big[ f^2 + g^2 \big] \, d\mu_\Gamma.
	\end{align}
	Next, we have by $d\mu_{g_z} = (1 + O(1)z) d\mu_\Gamma$:
	\begin{align} \label{eq:index.lower.wpar.full.3.6}
		& - 2 \eps \int_{-\eta}^\eta \int_{\Gamma} fg (\partial_z \mfh)(\cdot, 0) \overline \mbi_\eps'' \overline \mbh_\eps' \, d\mu_{g_z} \, dz \nonumber \\
		& \qquad = 2 \eps^2 \Big[ \int_\RR \mbh'' \mbi' \, dz \Big] \int_{\Gamma} (\partial_z \mfh)(\cdot, 0) fg \, d\mu_\Gamma + o(\eps^2) \int_\Gamma \big[ f^2 + g^2 \big] \, d\mu_\Gamma.
	\end{align}
	Finally, using
	\[ \int_{\RR} \mbh'' \mbh' \, dz = 0, \; \int_{\RR} z \mbh'' \mbh' \, dz = - \tfrac12 \energyunit, \]
	and \eqref{eq:mult.one.mean.curv.estimates}, and $d\mu_{g_z} = (1 + H_\Gamma z + O(z^2)) d\mu_\Gamma$, we have
	\begin{align} \label{eq:index.lower.wpar.full.3.7}
		& 2 \energyunit^{-1} \int_{-\eta}^\eta \int_{\Gamma} fg \mfh \overline \mbh_\eps'' (\overline \mbh_\eps' + \eps \mfh \overline \mbi_\eps') \, d\mu_{g_z} \, dz \nonumber \\
		& \qquad = 2 \energyunit^{-1} \int_{-\eta}^\eta \int_\Gamma fg (\mfh(\cdot, 0) + (\partial_z \mfh)(\cdot, 0) z + O(1) z^2) \overline \mbh_\eps'' \overline \mbh_\eps' (1 + H_\Gamma z + O(1) z^2) \, d\mu_\Gamma \, dz \nonumber \\
		& \qquad \qquad + 2 \energyunit^{-1} \eps \int_{-\eta}^\eta \int_\Gamma fg \mfh(\cdot, 0)^2 \overline \mbh_\eps'' \overline \mbi_\eps' \, d\mu_{g_z} \, dz \nonumber \\
		& \qquad = - \eps^2 \int_\Gamma (\mfh(\cdot, 0) H_\Gamma + (\partial_z \mfh)(\cdot, 0)) fg \, d\mu_\Gamma \nonumber \\
		& \qquad \qquad + 2 \energyunit^{-1} \eps^2 \Big[ \int_{\RR} \mbh'' \mbi' \, dz \Big] \int_{\Gamma} fg \mfh(\cdot, 0)^2 \, d\mu_{\Gamma} + o(\eps^2) \int_{\Gamma} \big[ f^2 + g^2 \big] \,  d\mu_\Gamma \nonumber \\
		& \qquad = - \tfrac12 \energyunit \eps^2 \int_\Gamma (H_\Gamma^2 + 2 \energyunit^{-1} (\partial_z \mfh)(\cdot, 0)) fg \, d\mu_\Gamma  \nonumber \\
		& \qquad \qquad + \tfrac12 \energyunit \eps^2 \Big[ \int_{\RR} \mbh'' \mbi' \, dz \Big] \int_{\Gamma} H_\Gamma^2 fg \, d\mu_\Gamma + o(\eps^2) \int_{\Gamma} \big[ f^2 + g^2 \big] \, d\mu_\Gamma.
	\end{align}
	We now collect terms. Up to error terms, the integrands that show up are:
	\[ \nabla_\Gamma f \cdot \nabla_\Gamma g, \; (|\sff_\Gamma|^2 + \ricc(\bm{n}_\Gamma, \bm{n}_\Gamma)) fg, \; (\partial_z \mfh)(\cdot, 0)^2 fg,  \text{ and } H_\Gamma^2 fg. \]
	Among them, $\nabla_\Gamma f \cdot \nabla_\Gamma g$ only appears in \eqref{eq:index.lower.wpar.full.1} with a coefficient of $\energyunit$, contributing
	\[ \energyunit \eps^2 \int_\Gamma \nabla_\Gamma f \cdot \nabla_\Gamma g \, d\mu_\Gamma \]
	to \eqref{eq:index.lower.wpar.full}. Next, $(|\sff_\Gamma|^2 + \ricc(\bm{n}_\Gamma, \bm{n}_\Gamma))fg$ only appears in \eqref{eq:index.lower.wpar.full.2.1}, \eqref{eq:index.lower.wpar.full.3.1}, with a total coefficient of $-\tfrac12 \energyunit - \tfrac12 \energyunit = -\energyunit$, contributing 
	\[ - \energyunit \int_\Gamma (|\sff_\Gamma|^2 + \ricc(\bm{n}_\Gamma, \bm{n}_\Gamma))fg \, d\mu_\Gamma \]
	to \eqref{eq:index.lower.wpar.full}. Next, $(\partial_z \mfh)(\cdot, 0) fg$ only appears in \eqref{eq:index.lower.wpar.full.3.1}, \eqref{eq:index.lower.wpar.full.3.2}, \eqref{eq:index.lower.wpar.full.3.6}, \eqref{eq:index.lower.wpar.full.3.7}, with a total coefficient of $-1 - 2 \Big[ \int_{\RR} \mbh'' \mbi' \, dz \Big] + 2 \Big[ \int_{\RR} \mbh'' \mbi \, dz \Big] - 1 = -2$, contributing
	\[ - 2 \int_\Gamma (\partial_z \mfh)(\cdot, 0) fg \, d\mu_\Gamma \]
	to \eqref{eq:index.lower.wpar.full}. Finally, $H_\Gamma^2 fg$ only appears in \eqref{eq:index.lower.wpar.full.2.1}, \eqref{eq:index.lower.wpar.full.2.2}, \eqref{eq:index.lower.wpar.full.2.3}, \eqref{eq:index.lower.wpar.full.3.2}, \eqref{eq:index.lower.wpar.full.3.3}, \eqref{eq:index.lower.wpar.full.3.4}, \eqref{eq:index.lower.wpar.full.3.5}, \eqref{eq:index.lower.wpar.full.3.7}, with a total coefficient of
	\begin{align*}
		& \tfrac12 \energyunit + \tfrac12 \energyunit \Big[ \int_{\RR} \mbh'' \mbi' \, dz \Big] - \tfrac12 \energyunit \Big[ \int_{\RR} \mbh'' \mbi' \, dz \Big] - \tfrac12 \energyunit \Big[ \int_{\RR} \mbh'' \mbi' \, dz \Big] \\
		& \qquad + \tfrac18 \energyunit^2 \Big[ \int_{\RR} W'''(\mbh) \mbi^2 \mbh'' \, dz \Big] + \tfrac18 \energyunit^2 \Big[ \int_{\RR} W''''(\mbh) \mbi^2 (\mbh')^2 \, dz \Big] + \tfrac14 \energyunit^2 \int_{\RR} W'''(\mbh) \mbi' \mbi \mbh' \, dz \Big] \\
		& \qquad - \tfrac12 \energyunit + \tfrac12 \energyunit \Big[ \int_{\RR} \mbh'' \mbi' \, dz \Big] \\
		& = \tfrac18 \energyunit^2 \int_{\RR} \big[ W'''(\mbh) \mbi^2 \mbh' \big]' \, dz = 0,
	\end{align*}
	thus not contributing to \eqref{eq:index.lower.wpar.full}. The lemma follows.
\end{proof}

\begin{proof}[Proof of Lemma \ref{lemm:index.lower.cross}]
	We have:
	\begin{align} \label{eq:index.lower.cross.full}
		& \cQ_{u,\eps}^\cU(f(\overline \mbh_\eps' + \eps \mfh \overline \mbi_\eps'), \omega) \nonumber \\
		& \qquad = \int_{-\eta}^\eta \int_\Gamma \Big[ -\eps \omega \cdot (\Delta_{g_z} + H_z \partial_z + \partial_z^2)(f(\overline \mbh_\eps' + \eps \mfh \overline \mbi_\eps')) \nonumber \\
		& \qquad \qquad \qquad + \eps^{-1} W''(u) \omega f (\overline \mbh_\eps' + \eps \mfh \overline \mbi_\eps') \Big] \, d\mu_{g_z} \, dz \nonumber \\
		& \qquad = \int_{-\eta}^\eta \int_\Gamma \Big[ \eps \nabla_{g_z} \omega \cdot \nabla_{g_z} (f(\overline \mbh_\eps' + \eps \mfh \overline \mbi_\eps')) - H_z \omega f (\overline \mbh_\eps'' + \eps \mfh \overline \mbi_\eps'' + \eps (\partial_z \mfh) \overline \mbi_\eps') \nonumber \\
		& \qquad \qquad \qquad + \eps^{-1} \omega f \big( W''(u)(\overline \mbh_\eps' + \eps \mfh \overline \mbi_\eps') - \overline \mbh_\eps''' - \eps \mfh \overline \mbi_\eps''' - 2 \eps^2 (\partial_z \mfh) \overline \mbi_\eps'' - \eps^3 (\partial_z^2 \mfh) \overline \mbi_\eps') \big) \Big] \, d\mu_{g_z} \, dz.
	\end{align}
	We estimate the terms in \eqref{eq:index.lower.cross.full} one by one. We have, by \eqref{eq:mult.one.h.phi.horizontal.estimates}:
	\begin{align} \label{eq:index.lower.cross.full.1.intermediate}
		& \int_{-\eta}^\eta \int_\Gamma \eps \nabla_{g_z} \omega \cdot \nabla_{g_z} (f(\overline \mbh_\eps' + \eps \mfh \overline \mbi_\eps')) \, d\mu_{g_z} \nonumber \\
		& \qquad = \int_{-\eta}^\eta \int_\Gamma \eps (1 + O(1)z) \nabla_\Gamma \omega \cdot \big[ (1 + O(1) z) (\nabla_\Gamma f) (\overline \mbh_\eps' + \eps \mfh \overline \mbi_\eps') \nonumber \\
		& \qquad \qquad \qquad - \eps^{-1} f (\overline \mbh_\eps'' \nabla_{g_z} h + \eps \mfh \overline \mbi_\eps'' \nabla_{g_z} h - \eps^2 (\nabla_{g_z} \mfh) \overline \mbi_\eps') \big] \, (1 + O(1) z) d\mu_\Gamma \, dz \nonumber \\
		& \qquad = \int_{-\eta}^\eta \int_\Gamma \eps (\nabla_\Gamma \omega \cdot \nabla_\Gamma f) \overline \mbh_\eps' \, d\mu_\Gamma \, dz \nonumber \\
		& \qquad \qquad + o(\eps^2) \int_\Gamma \big[ |\nabla_\Gamma f|^2 + f^2 \big] \, d\mu_\Gamma + o(1) \int_{\cU} \big[ \eps |\nabla \omega|^2 + \eps^{-1} \omega^2 \big] \, d\mu_g.
	\end{align}
	In the last step, we used Cauchy--Schwarz. Note that $\nabla_\Gamma w^\perp \cdot \nabla_\Gamma w^\parallel = g^{ij}_\Gamma \partial_{y_i} w^\parallel \partial_{y_j} w^\perp$, whose two first factors are independent of $z$, and that, since $\omega$ satisfies \eqref{eq:index.lower.wperp} for all $y$:
	\[ \int_{-\eta}^{\eta} (\partial_{y_i} \omega) \overline \mbh_\eps' \, dz	 = \eps^{-1} \int_{-\eta}^\eta (\partial_{y_i} h) \omega \overline \mbh_\eps'' \, dz. \]
	Using this to estimate the first term in \eqref{eq:index.lower.cross.full.1.intermediate} via  \eqref{eq:mult.one.h.phi.horizontal.estimates} and Cauchy--Schwarz, we deduce:
	\begin{align} \label{eq:index.lower.cross.full.1}
		& \int_{-\eta}^\eta \int_\Gamma \eps \nabla_{g_z} \omega \cdot \nabla_{g_z} (f(\overline \mbh_\eps' + \eps \mfh \overline \mbi_\eps')) \, d\mu_{g_z} \nonumber \\
		& \qquad = o(\eps^2) \int_\Gamma \big[ |\nabla_\Gamma f|^2 + f^2 \big] \, d\mu_\Gamma + o(1) \int_{\cU} \big[ \eps |\nabla \omega|^2 + \eps^{-1} \omega^2 \big] \, d\mu_g.
	\end{align}
	By Cauchy--Schwarz again, the boundedness of $H_z$, and exponential decay of $\overline{\mbi}_\eps'$, $\overline{\mbi}_\eps''$:
	\begin{align} \label{eq:index.lower.cross.full.2}
		& \int_{-\eta}^\eta \int_\Gamma H_z \omega f (\eps \mfh \overline{\mbi}_\eps'' + \eps (\partial_z \mfh) \overline{\mbi}_\eps') \, d\mu_{g_z} \, dz \nonumber \\
		& \qquad = o(\eps^2) \int_{\Gamma} f^2 \, d\mu_\Gamma + o(1) \int_{\cU} \eps^{-1} \omega^2 \, d\mu_g.
	\end{align}
	Likewise:
	\begin{align} \label{eq:index.lower.cross.full.3}
		& \int_{-\eta}^\eta \int_\Gamma \eps^{-1} \omega f (2 \eps^2 (\partial_z \mfh) \overline{\mbi}_\eps'' + \eps^3 (\partial_z^2 \mfh) \overline{\mbi}_\eps') \, d\mu_{g_z} \, dz \nonumber \\
		& \qquad = o(\eps^2) \int_{\Gamma} f^2 \, d\mu_\Gamma + o(1) \int_{\cU} \eps^{-1} \omega^2 \, d\mu_g.
	\end{align}
	We are left trying to estimate
	\begin{align} \label{eq:index.lower.cross.full.4}
		& \int_{-\eta}^\eta \int_\Gamma \eps^{-1} \omega f \Big[ - \eps H_z \overline{\mbh}_\eps'' + W''(u)(\overline{\mbh}_\eps' + \eps \mfh \overline{\mbi}_\eps') - \overline{\mbh}_\eps''' - \eps \mfh \overline{\mbi}_\eps''' \Big] \, d\mu_{g_z} \, dz \nonumber \\
		& \qquad = \int_{-\eta}^\eta \int_\Gamma \eps^{-1} \omega f \Big[ - \eps H_z \overline{\mbh}_\eps'' + W''(\overline{\mbh}_\eps)(\overline{\mbh}_\eps' + \eps \mfh \overline{\mbi}_\eps') + W'''(\overline{\mbh}_\eps)(\eps \mfh \overline{\mbi}_\eps + \hat \phi) \overline{\mbh}_\eps' \nonumber \\
		& \qquad \qquad \qquad \qquad \qquad - W''(\overline{\mbh}_\eps) \overline{\mbh}_\eps' - \eps \mfh W''(\overline{\mbh}_\eps) \overline{\mbi}_\eps' - \eps \mfh W'''(\overline{\mbh}_\eps) \overline{\mbh}_\eps' \overline{\mbi}_\eps + 2 \energyunit^{-1} \eps \mfh \overline{\mbh}_\eps'' \nonumber \\
		& \qquad \qquad \qquad \qquad \qquad + O(\eps^2) (\overline{\mbh}_\eps' + |\overline{\mbi}_\eps'|) + O(\eps^3) \Big] \, d\mu_{g_z} \, dz \nonumber \\
		& \qquad = \int_{-\eta}^\eta \int_\Gamma \Big[ O(\eps)(\overline{\mbh}_\eps' + |\overline{\mbi}_\eps'|) + O(\eps^2) \Big] \omega f \, d\mu_{g_z} \, dz \nonumber \\
		& \qquad = o(\eps^2) \int_{\Gamma} f^2 \, d\mu_\Gamma + o(1) \int_{\cU} \eps^{-1} \omega^2 \, d\mu_g.
	\end{align}
	Above, we used $u = \overline{\mbh}_\eps + \eps \mfh \overline{\mbi}_\eps + \hat \phi$ to expand $W''(u)$, \eqref{eq:mult.one.phihat.estimate} to estimate $\hat \phi$; we expanded $H_z = H_\Gamma + O(1) z$ and used \eqref{eq:mult.one.mean.curv.estimates.more.refined} to bound $H_\Gamma - 2 \energyunit^{-1} \mfh$; and, in the last step we used Cauchy--Schwarz. The lemma follows by combining 
	\eqref{eq:index.lower.cross.full}, \eqref{eq:index.lower.cross.full.1.intermediate}, \eqref{eq:index.lower.cross.full.1}, \eqref{eq:index.lower.cross.full.2}, \eqref{eq:index.lower.cross.full.3}, \eqref{eq:index.lower.cross.full.4}.
\end{proof}

\begin{proof}[Proof of Lemma \ref{lemm:index.lower.wperp}]
	This is the same as in \cite[Lemma 5.8]{ChodoshMantoulidis:multiplicity-one}. It is a consequence of the strict stability of $-\frac{d}{dt^2} + W''(\mbh)$ once we work orthogonally to its kernel using \eqref{eq:index.lower.wperp}.
\end{proof}

\section{Proof of Theorem \ref{theo:higher.dim}}

\begin{proof}[Proof of (a)]
	This is a consequence of Theorem \ref{theo:index.upper}.
\end{proof}

\begin{proof}[Proof of (b)]
	If Section \ref{sec:mult.one} applies, then the $C^{2,\alpha}$ convergence follows from  \eqref{eq:mult.one.mean.curv.estimates}, Remark \ref{rema:mult.one.mean.curv.estimates.tau}, and Schauder theory. To that end, it suffices to arrange \eqref{eq:mult.one.assumptions.i}, \eqref{eq:mult.one.assumptions.ii}, \eqref{eq:mult.one.assumptions.iii}. This is done as in \cite[Theorem 3.4]{ChodoshMantoulidis:multiplicity-one}, provided we can arrange for \eqref{eq:mult.one.assumptions.ii} (this is where $n=3$ and stability were used in \cite{ChodoshMantoulidis:multiplicity-one}). If \eqref{eq:mult.one.assumptions.ii} failed for $\eps \to 0$, we could take a sequence of counterexamples $(u_i, \eps_i, \mfh_i)$ satisfying \eqref{eq:mult.one.assumptions.i}, with $\eps_i \to 0$ and $\eps_i |\nabla u_i(p_i)| \to 0$ for some $p_i \in \{ |u_i| < 1-\beta_0 \}$. Passing to a subsequence, $u_i(\eps_i(\cdot - p_i))$ would converge to a solution of $\Delta u = W'(u)$ on $\RR^n$ with $\nabla u(\bm{0}) = \bm{0}$. This solution would also have to have density $1$ at infinity, by virtue of monotonicity. Thus, by \cite{Wang:Allard}, it would have to be a rotation of the heteroclinic solution, which has a nonzero gradient, a contradiction.
\end{proof}

\begin{proof}[Proof of (c)]
	If $\Theta(V, \cdot) \equiv 1$ on $\support \Vert V \Vert$, then $\support \Vert V \Vert$ is smooth by \cite{HutchinsonTonegawa00} and Allard's theorem \cite{Simon83}. Therefore, Section \ref{sec:index.lower} applies and the result follows by Theorem \ref{theo:index.lower}.
\end{proof}

\section{Open questions} \label{sec:open}

Some interesting directions in the variational study of multiplicity-one solutions of \eqref{eq:ac.pde.h} that merit further investigation:

\begin{enumerate}
	\item \textbf{Self-tangencies}. What can be said about the index of $V$ without treating self-tangencies along smooth pieces as parts of the ``fixed'' singular set? Can one devise settings in which self-tangencies do not occur? (cf. \cite{White:generic-transversality}.)
	\item \textbf{Isoperimetric variational problem}. The index and nullities considered in this paper are the variational quantities that one can control through a min-max construction of critical points that fixes $\mfh$. See Remark \ref{rema:isoperimetric.index.h}. However, one may instead wish to fix the enclosed volumes, thus giving up exact control of $\mfh$. See \cite{PacardRitore03, BenciNardullOsorioPiccione}. This alternative setting can be referred to as the \textit{isoperimetric} (i.e., fixed volume) setting. The regularity and asymptotics from Section \ref{sec:mult.one} can apply to the isoperimetric setting too. However, one needs to modify Theorems \ref{theo:index.upper}, \ref{theo:index.lower} to fit into the isoperimetric setting. Modifications of both theorems include subtle points. 
	\item \textbf{Uniqueness}. When $\mfh \equiv 0$, it was shown in \cite{GuaracoMarquesNeves} that multiplicity-one critical points $(u, \eps, 0)$ near nondegenerate minimal surfaces coincide with those constructed by Pacard \cite{Pacard12} and, a posteriori, must also coincide with those in the earlier work of Pacard--Ritore \cite{PacardRitore03}. The proof used the sharp asymptotics derived by Wang--Wei (\cite{WangWei}). Given the sharp asymptotics for the general $\mfh$ setting now obtained in Section \ref{sec:mult.one}, one should be able to prove a corresponding uniqueness theorem. 
\end{enumerate}

\appendix

\section{Derivation of \eqref{eq:mult.one.h.pde} and  \eqref{eq:mult.one.mean.curv.estimates.more.refined}}

\label{app:mult.one.h.pde}

In what follows, \eqref{eq:mult.one.sff.bounds} gets used repeatedly though implicitly when obtaining $O_{1,0,\alpha,\eps}$ bounds.

We project \eqref{eq:mult.one.phi.pde} onto $\Gamma$ by fixing $y \in B_{19}^\Gamma$, dotting with $\overline \mbh_\eps'(y, z)$ and integrating over $z$. We start with the left hand side. We differentiate $\phi \perp \overline \mbh_\eps'$ along $y$ and use \eqref{eq:mult.one.ddt.connection} to get:
\begin{align} \label{eq:h.equation.lhs.1}
	\int_\RR \eps^2 (\Delta_{g_z} \phi) \overline \mbh_\eps' \, dz
		& = \int_\RR \eps^2 (\Delta_\Gamma \phi) \overline \mbh_\eps' \, dz + \int_\RR  \eps^2 (\Delta_{g_z} \phi - \Delta_\Gamma \phi) \overline \mbh_\eps' \, dz  \nonumber \\
		& = - \int_\RR \eps^2 \phi (\Delta_\Gamma \overline \mbh_\eps') \, dz - \int_\RR \eps^2 (\nabla_\Gamma \phi) \cdot (\nabla_\Gamma \overline \mbh_\eps') \, dz \nonumber \\
		& \qquad + \int_\RR  \eps^2 (\Delta_{g_z} \phi - \Delta_\Gamma \phi) \overline \mbh_\eps' \, dz \nonumber \\
		& = \int_\RR \phi (\eps (\Delta_\Gamma h) \overline \mbh_\eps'' - |\nabla_\Gamma h|^2 \overline \mbh_\eps''') \, dz + \int_\RR \eps (\nabla_\Gamma \phi) \cdot (\nabla_\Gamma h) \overline \mbh_\eps'' \, dz  \nonumber \\
		& \qquad + \int_\RR  \eps^2 (\Delta_{g_z} \phi - \Delta_\Gamma \phi) \overline \mbh_\eps' \, dz \nonumber \\
		& = \eps (\Delta_\Gamma h) \int_\RR \phi \overline \mbh_\eps'' \, dz +  \eps \cdot  O_{1,0,\alpha,\eps}(\phi) (O_{1,0,\alpha,\eps}(\nabla_\Gamma h))^2 \nonumber \\
		& \qquad + \eps \cdot O_{1,0,\alpha,\eps}(\eps \nabla_\Gamma \phi)  \cdot O_{1,0,\alpha,\eps}(\nabla_\Gamma h) \nonumber \\
		& \qquad + \eps^2 \cdot O_{1,0,\alpha,\eps}(\eps^2 \nabla^2_\Gamma \phi, \eps \nabla_\Gamma \phi).
\end{align}
Next, integrating by parts yields and using $\phi \perp \overline \mbh_\eps'$ again:
\begin{align} \label{eq:h.equation.lhs.2}
	\int_\RR \eps^2 H_z (\partial_z \phi) \overline \mbh_\eps' \, dz
		& = - \int_\RR \eps H_z \phi \overline \mbh_\eps'' \, dz - \int_\RR \eps^2 (\partial_z H_z) \phi \overline \mbh_\eps' \, dz \nonumber \\
		& = - \int_\RR \eps (H_\Gamma \phi + O(1) z) \overline \mbh_\eps'' \, dz - \int_\RR \eps^2 (\partial_z H_z) \phi \overline \mbh_\eps' \, dz \nonumber \\
		& = - \eps H_\Gamma \int_\RR \phi \overline \mbh_\eps'' \, dz + \eps^3 \cdot O_{1,0,\alpha,\eps}(\phi).
\end{align}
Next, integrating by parts twice yields:
\begin{equation} \label{eq:h.equation.lhs.3}
	\int_\RR \big[ \eps^2 (\partial_z^2 \phi) - W''(\overline \mbh_\eps) \phi \big] \overline \mbh_\eps' \, dz = \eps^3 \cdot O_{1,0,\alpha,\eps}(\phi).
\end{equation}
We move on to the right hand side of \eqref{eq:mult.one.phi.pde}. We have:
\begin{align} \label{eq:h.equation.rhs.1}
	\int_\RR \eps \mfh \overline \mbh_\eps' \, dz 
		& = \int_\RR \eps (\mfh(\cdot, 0) + (\partial_z \mfh)(\cdot, 0) z + O_{C^{1,\alpha}_\eps}(1) z^2) \overline \mbh_\eps' \, dz \nonumber \\
		& = 2 \eps^2 \mfh(\cdot, 0) + 2 \eps^2 (\partial_z \mfh(\cdot, 0)) h + O_{1,0,\alpha,\eps}(\eps^4).
\end{align}
Next:
\begin{equation} \label{eq:h.equation.rhs.2}
	\int_\RR \eps (H_\Gamma - \Delta_\Gamma h) (\overline \mbh_\eps')^2 \, dz = \eps^2 \energyunit (H_\Gamma - \Delta_\Gamma h).
\end{equation}
Next:
\begin{equation} \label{eq:h.equation.rhs.3}
	\int_\RR \eps (|\sff_\Gamma|^2 + \ricc(\bm{n}_\Gamma, \bm{n}_\Gamma)) z \overline \mbh_\eps' \, dz = 2 \eps^2 (|\sff_\Gamma|^2 + \ricc(\bm{n}_\Gamma, \bm{n}_\Gamma)) h.
\end{equation}
For now, we estimate:
\begin{equation} \label{eq:h.equation.rhs.4}
	\int_\RR \tfrac12 W'''(\overline \mbh_\eps) \phi^2 \overline \mbh_\eps' \, dz = \eps \cdot (O_{1,0,\alpha,\eps}(\phi))^2,
\end{equation}
though we will refine this estimate later once we get a more precise form of $\phi$. Finally:
\begin{align} \label{eq:h.equation.rhs.5}
	& \int_\RR \big[ (O_{1,0,\alpha,\eps}(\phi))^3 + O_{1,0,\alpha,\eps}(\eps \nabla_\Gamma^2 h, \nabla_\Gamma h) z \overline \mbh_\eps' + (O_{1,0,\alpha,\eps}(\nabla_\Gamma h))^2 \overline \mbh_\eps'' + O_{1,0,\alpha,\eps}(\eps^3) \big] \overline \mbh_\eps' \, dz \nonumber \\
	& \qquad = \eps \cdot (O_{1,0,\alpha,\eps}(\phi))^3 + \eps^2 \cdot O_{1,0,\alpha,\eps}(\eps \nabla_\Gamma^2 h, \nabla_\Gamma h) + \eps \cdot (O_{1,0,\alpha,\eps}(\nabla_\Gamma h))|^2 + O_{1,0,\alpha,\eps}(\eps^4).
\end{align}
At this point, \eqref{eq:mult.one.h.pde} follows from combining \eqref{eq:h.equation.lhs.1}, \eqref{eq:h.equation.lhs.2}, \eqref{eq:h.equation.lhs.3}, \eqref{eq:h.equation.rhs.1}, \eqref{eq:h.equation.rhs.2}, \eqref{eq:h.equation.rhs.3}, \eqref{eq:h.equation.rhs.4}, \eqref{eq:h.equation.rhs.5}, and finally estimating $h$ by $\phi$ as in \cite[Lemma 9.6]{WangWei}.

Finally, let us assume we have a more refined ansatz for $\phi$, namely:
\[ \phi = \hat \phi + \eps \mfh \overline \mbi_\eps \]
where $\hat \phi = O_{1,0,\alpha,\eps}(\eps^2)$. Then, we can replace \eqref{eq:h.equation.rhs.4} by
\begin{align} \label{eq:h.equation.rhs.4.refined} 
	& \int_\RR \tfrac12 W'''(\overline \mbh_\eps) \phi^2 \overline \mbh_\eps' \, dz \nonumber \\
	& \qquad = \eps^2 \int_\RR \tfrac12 W'''(\overline \mbh_\eps) \mfh^2 \overline \mbi_\eps^2 \overline \mbh_\eps' \, dz + O_{1,0,\alpha,\eps}(\eps^4) \nonumber \\
	& \qquad = \eps^2 \int_\RR \tfrac12 W'''(\overline \mbh_\eps) (\mfh(\cdot, 0) + O_{1,0,\alpha,\eps}(1) z)^2 \overline \mbi_\eps^2 \overline \mbh_\eps' \, dz + O_{C^{0,\alpha}_\eps}(\eps^4) \nonumber \\
	& \qquad = \eps^2 \mfh(\cdot, 0)^2 \int_\RR \tfrac12 W'''(\overline \mbh_\eps) \overline \mbi_\eps^2 \overline \mbh_\eps' \, dz + O_{1,0,\alpha,\eps}(\eps^4) = O_{1,0,\alpha,\eps}(\eps^4),
\end{align}
where in the last step we've used \eqref{eq:mult.one.h.phi.estimates} and the fact that, by parity,
\[ \int_\RR W'''(\mbh) \mbi^2 \mbh' \, dz = 0. \]
Now, \eqref{eq:mult.one.mean.curv.estimates.more.refined} follows from the same equations, with  \eqref{eq:h.equation.rhs.4.refined} replacing \eqref{eq:h.equation.rhs.4}.

\section{Derivation of \eqref{eq:mult.one.h.phi.estimates}, \eqref{eq:mult.one.h.phi.horizontal.estimates}, \eqref{eq:mult.one.phihat.estimate}} 

\label{app:estimate.orthogonal.to.h}

This section is meant to simplify and condense the exposition in \cite[Sections 11-13]{WangWei} by exploiting the multiplicity-one setting. It is borrowed from collaborative notes written with O. Chodosh. In this appendix we will assume, without loss of generality, that $W''(\pm 1) = 2$.

\begin{lemm}\label{lemm:1d-bernstein-linearized-Hprime}
Consider $w \in C^{2}(\RR^{n})$ and $f \in C^{0}(\RR^{n-1})$ so that, for $(y, z) \in \RR^{n-1} \times \RR = \RR^n$,
\[
\Delta_{\RR^{n-1}} w(y,z) + \partial_{z}^{2} w(y,z) - W''(\mbh(z)) w(y,z) = f(y)\mbh'(z).
\]
Then, there is some $c \in C^{2}(\RR^{n-1})$ so that $w=c(y)\mbh'(z)$. 
\end{lemm}
\begin{proof}
We mimic \cite[Lemma 3.7]{Pacard12}. Write
\[
w(y,z) = c(y) \mbh'(z) + \bar w(y,z)
\]
where $\int_{-\infty}^{\infty} w(z,y) \mbh'(z) dz = 0$ for all $y\in \mathbb{R}^{n-1}$. We thus find that
\[
\mbh'(z) \Delta_{\RR^{n-1}} c(y) + (\partial^{2}_{z} \bar w(y,z) - W''(\mbh(z)) \bar w(y,z) + \Delta_{\RR^{n-1}} \bar w(y,z)) = f(y)\mbh'(z). 
\]
Multiplying by $\mbh'(z)$ and integrating, we find that $\Delta_{\RR^{n-1}} c(y) = f(y)$,
 and so 
\[
\partial^{2}_{z} \bar w(y,z) - W''(\mbh(z)) \bar w(y,z) + \Delta_{\RR^{n-1}} \bar w(y,z) = 0. 
\]
At this point, the proof that $\bar w = 0$ is identical to \cite[Lemma 3.7]{Pacard12}.
\end{proof}

\begin{lemm}\label{lemm:Hprime-perp-estimate}
Fix $\sigma\in(0,1)$. Then, we can choose $L>0$ and $C>0$ depending on $\sigma$, and $K>0$ sufficiently large depending only on $W$ with the following property. Suppose that 
\begin{equation}\label{eq:C0-est-eqn-general-form}
\eps^{2} (\Delta_{\Gamma} \psi + \partial^{2}_{z}\psi) - W''(\overline \mbh_\eps) \psi = \eps f_{1}(y) \overline \mbh_{\eps}'(y,z) + f_{2}(y,z) + \eps D_{i} f_{3}^{(i)}(y,z)
\end{equation}
on $B_{r+2L\eps}^{\Gamma} \times \mathbf{I}_\eps$. Then, for $\eps>0$ sufficiently small, either
\[
 \Vert \psi \Vert_{C^{0}(B_{r}^\Gamma \times \mathbf{I}_{\eps})} \leq  2 \mbh'(0) \energyunit^{-1} \sup_{y\in B^{\Gamma}_{r+2L\eps}} \left| \int_{-\eps K}^{\eps K} \psi(y,z) \overline \mbh'_{\eps}(z-h(y)) \, dz \right|
\]
or
\begin{align*}
	\Vert \psi\Vert_{C^{0}(B_{r}^{\Gamma} \times \mathbf{I}_\eps)} 
	&  \leq \sigma \left( \Vert \psi \Vert_{C^{1}_{\eps}(B_{r+2L\eps}^{\Gamma} \times \mathbf{I}_\eps)} + \Vert f_{1} \Vert_{C^{0,\alpha}_{\eps}(B_{r+2L\eps}^{\Gamma})} \right) \\
	& \qquad + C \left( \Vert f_{2}\Vert_{C^{0}(B_{r+2L\eps}^{\Gamma} \times \mathbf{I}_\eps)}  + \Vert \vec{f}_{3} \Vert_{C^{0,\alpha}_{\eps}(B_{r+2L\eps}^{\Gamma} \times \mathbf{I}_\eps)} + \Vert \psi \Vert_{C^0(B_{r}^{\Gamma} \times \mathbf{J}_{\eps,L})} \right),
\end{align*}
where $\mathbf{J}_{\eps,L}$ denotes the points of $\mathbf{I}_\eps$ that are within $\eps L$ of $\partial \mathbf{I}_\eps$.
\end{lemm}
\begin{proof}
First, choose $\tilde \chi : B_{r+2L\eps}^{\Gamma}\to[0,1]$ a cutoff function that is $1$ on $B_{r}^{\Gamma}$ and has support in $B_{r+ L\eps}^{\Gamma}$. We can arrange so that $\eps L |\nabla_{\Gamma} \tilde \chi| + \eps^{2}L^{2} |\nabla^{2}_{\Gamma} \tilde \chi|^{2} = O(1)$. Now, by replacing $\psi$ by $\tilde \chi \psi$ and absorbing the resulting error terms into $f_{2}$, it is clear that it suffices to prove that
\begin{align}\label{eq:Hperp-est-cutoff-estimate}
	\Vert \psi\Vert_{C^{0}(B_{r}^{\Gamma} \times \mathbf{I}_\eps)} 
	& \leq \sigma \Vert f_{1} \Vert_{C^{0,\alpha}_{\eps}(B_{r+2L\eps}^{\Gamma} \times \mathbf{I}_\eps)} \nonumber \\
	& \qquad + C \left( \Vert f_{2}\Vert_{C^{0}(B_{r+2L\eps}^{\Gamma} \times \mathbf{I}_\eps)}  + \Vert \vec{f}_{3} \Vert_{C^{0,\alpha}_{\eps}(B_{r+2L\eps}^{\Gamma} \times \mathbf{I}_\eps)} + \Vert \psi \Vert_{C^0(B_r^\Gamma \times \mathbf{J}_{\eps,L})} \right)
\end{align}
assuming that $\psi$ is supported in $B^{\Gamma}_{r+\frac12 L\eps}\times\mathbf{I}_{\eps}$ and satisfies \eqref{eq:C0-est-eqn-general-form} and 
\begin{equation}\label{eq:Hperp-est-psi-Hprime-est}
\sup_{y\in B^{\Gamma}_{r+2L\eps}} \left| \int_{-\eps K}^{\eps K} \psi(y,z) \overline \mbh'_{\eps}(z-h(y)) \, dz \right| < \frac 1 2 \mbh'(0)^{-1} \energyunit  \Vert \psi \Vert_{C^{0}(B^\Gamma_{r} \times \mathbf{I}_{\eps})}.
\end{equation}
Assume, for contradiction, that \eqref{eq:Hperp-est-cutoff-estimate} fails. Then, there are $C,L\to\infty$ as $\eps\to 0$ so that
\begin{align*}
\Vert \psi\Vert_{C^{0}(B_{r}^{\Gamma} \times \mathbf{I}_\eps)} 
& \geq \sigma \Vert f_{1} \Vert_{C^{0,\alpha}_{\eps}(B_{r+2L\eps}^{\Gamma} \times \mathbf{I}_\eps)} \\
& \qquad + C \left(\Vert f_{2}\Vert_{C^{0}(B_{r+2L\eps}^{\Gamma} \times \mathbf{I}_\eps)}  + \Vert \vec{f}_{3} \Vert_{C^{0,\alpha}_{\eps}(B_{r+2L\eps}^{\Gamma} \times \mathbf{I}_\eps)} + \Vert \psi \Vert_{C^0(B_r^\Gamma \times \mathbf{J}_{\eps,L})} \right).
\end{align*}
Choose $\bar x = (\bar y,\bar z) \in \overline{B_{r}^{\Gamma}\times \mathbf{I}_{\eps}}$ attaining $\Vert \psi\Vert_{C^{0}(B_{r}^{\Gamma} \times \mathbf{I}_\eps)}$. Set $\tilde z = \eps^{-1}\bar z$. We first assume that $\tilde z \to \hat z$ as $\eps \to 0$. The case that $\tilde z$ is unbounded as $\eps\to 0$ follows from a similar, but simpler argument, as we describe below. Dividing the equation by $\pm\Vert \psi\Vert_{C^{0}(B_{r}^{\Gamma} \times \mathbf{I}_\eps)}$ and rescaling around $\bar x$ to scale $\eps$ (labeling rescaled quantities with a tilde), we find that $\tilde \psi(0) = 1$, $\Vert \tilde \psi\Vert_{C^{0}(B_{L})} = 1$, 
\[
\Delta_{\tilde \Gamma} \tilde \psi + \partial^{2}_{z}\tilde\psi - W''(\tilde {\overline \mbh}) \tilde \psi =  \tilde f_{1}(y) \tilde{\overline \mbh}'(z-\tilde z-\eps^{-1}\tilde h(y)) + \tilde f_{2}(y,z) +  D_{i} \tilde f_{3}^{(i)}(y,z),
\]
on $B_{L}$, and finally
\[
\Vert \tilde f_{1} \Vert_{C^{0,\alpha}(B_{L})} \leq \sigma^{-1}, \text{ and } \Vert \tilde f_{2}\Vert_{C^{0}(B_{L})}  + \Vert \tilde {f}_{3} \Vert_{C^{0,\alpha}_{\eps}(B_{L})} = o(1).
\]
Hence, $\tilde f_{2}\to 0$ in $C^{0}(B_{L})$ and $\tilde f_{3}^{(i)}\to 0$ in $C^{0,\alpha}(B_{L})$. Moreover, $\tilde f_{1}$ is bounded in $C^{0,\alpha}(B_{L})$. We can thus find $\hat f_{1}\in C^{0,\alpha}(\RR^{n-1})$ so that $\tilde f_{1}\to\hat f_{1}$ in $C^{0,\alpha'}_{\textrm{loc}}(\RR^{n-1})$ for $\alpha' < \alpha$. 

Similarly, by $C^{1,\alpha}$-Schauder estimates  we see that $\tilde \psi$ is uniformly bounded in $C^{1,\alpha}$ on compact subsets of $\RR^{n}$. Thus, there is $\hat \psi\in C^{1,\alpha}_{\textrm{loc}}(\RR^{n})\cap L^{\infty}(\RR^{n})$ so that $\tilde\psi \to\hat \psi$ in $C^{1,\alpha'}_{\textrm{loc}}(\RR^{n})$. Integrating by parts against a test function, we see that $\hat \psi$ weakly solves
\[
\Delta_{\RR^{n-1}} \hat \psi + \partial^{2}_{z} \hat \psi - W''(\mbh(z-\tilde z)) \hat \psi =  \hat f_{1}(y) \mbh'(z-\tilde z). 
\]
Schauder theory implies that $\hat \psi \in C^{2,\alpha}(\RR^{n})$. By Lemma \ref{lemm:1d-bernstein-linearized-Hprime}, we have that $\hat \psi = c(y) \mbh'(z-\hat z)$. Because $\hat \psi(0) = 1 = \Vert \hat\psi\Vert_{L^{\infty}(\RR^{n})}$, we see that $\hat z =0$ and $c(0) = \mbh'(0)^{-1}$. Thus, we see that
\[
\int_{-K}^{K}\hat \psi(0,z) \mbh'(z) dz = \mbh'(0)^{-1} \energyunit + O(e^{-\sqrt{2}K})
\]
Returning to $\psi$, we thus find that
\[
\sup_{y\in B^{\Gamma}_{r+2L\eps}} \left| \int_{-\eps K}^{\eps K} \psi(y,z) \overline \mbh'_{\eps}(z-h(y)) dz \right| \geq \left(\mbh'(0)^{-1} \energyunit + O(e^{-\sqrt{2}K}) + o(1) \right)\Vert \psi \Vert_{C^{0}(B_{r}^\Gamma \times \mathbf{I}_{\eps})}
\]
as $\eps \to 0$. Taking $K$ sufficiently large this contradicts \eqref{eq:Hperp-est-psi-Hprime-est} for $\eps$ sufficiently small.

Finally, if the case that $\tilde z\to\infty$, then repeating the same rescaling as above (but using $\mbh(t)\to \pm 1$ as $t\to\pm\infty$), we find $\hat \psi \in C^{2,\alpha}_{\textrm{loc}}(\RR^{n})\cap L^{\infty}(\RR^{n})$, with $\hat \psi(0) = 1$ and so that
\[
\Delta_{\RR^{n}}\hat \psi - W''(\pm 1) \hat \psi = 0.
\]
Because $\hat \psi$ attains its maximum at $0$, we see that $\hat \psi \equiv 0$, a contradiction.
\end{proof}

We note how the first alternative of Lemma \ref{lemm:Hprime-perp-estimate} can never apply to $\phi$, provided $K$ is chosen sufficiently large. Indeed, it follows from \eqref{eq:mult.one.phi.orthogonal} that
\begin{equation} \label{eq:mult.one.phi.first.alternative}
	\left| \int_{-\eps K}^{\eps K} \phi(y, z) \overline \mbh'_\eps(y, z) \, dz \right| = \left| \int_{\mathbf{I}_{\eps} \setminus [-\eps K, \eps K]} \phi(y, z) \overline \mbh'_\eps(y, z) \, dz \right| \leq C e^{-\sqrt{2}K} \Vert \phi(y,  \cdot) \Vert_{C^0(\mathbf{I}_\eps)}.
\end{equation}
Therefore, for sufficiently large (but fixed) choices of $K$, the second alternative of Lemma \ref{lemm:Hprime-perp-estimate} must always hold when $\psi = \phi$.

Let us use this fact to prove \eqref{eq:mult.one.h.phi.estimates}. We first note that \eqref{eq:mult.one.phi.pde} and \eqref{eq:mbh.times.polynomial} imply
\begin{align} \label{eq:mult.one.phi.pde.Oeps}
	\eps^2 \Delta \phi - W''(\overline \mbh_\eps) \phi 
		& = - \eps(H_\Gamma - \Delta_\Gamma h) \overline{\mbh}_\eps' + O_{1,0,\alpha,\eps}(\eps \mfh, \eps^2) + (O_{1,0,\alpha,\eps}(\phi))^2  \nonumber \\
		& \qquad + (O_{1,0,\alpha,\eps}(\eps \nabla_\Gamma^2 h,   \nabla_\Gamma h))^2 \nonumber \\
		& = - \eps(H_\Gamma - \Delta_\Gamma h) \overline{\mbh}_\eps' + O_{1,0,\alpha,\eps}(\eps) + (O_{1,0,\alpha,\eps}(\eps^2 \nabla^2_\Gamma \phi, \eps \nabla_\Gamma \phi, \phi))^2,
\end{align}
where the second equation follows from the first from our bounds on the prescribed function $\mfh$ our ability to control the height adjustment $h$ in terms of $\phi$ (\cite[Lemma 9.6]{WangWei}).

Fix $\sigma \in (0, 1)$. We apply Lemma \ref{lemm:Hprime-perp-estimate} in $B_{19}^\Gamma \times \mathbf{I}_\eps$ to get a $C^0$ estimate on $\phi$ in $B_{19 - 2 \eps L}^\Gamma \times \mathbf{I}_\eps$ (using \eqref{eq:mult.one.connection.estimates} to treat $\eps^2 (\Delta - \Delta_\Gamma - \partial_z^2) \phi$ as a right hand side term), which can be enlarged to a $C^0$ estimate on $B_{19 - 2 \eps L}^\Gamma \times (-1, 1)$ with at most an $O(\eps)$ error using the decay of $\phi$ off $\Gamma$. Then use Schauder theory on \eqref{eq:mult.one.phi.pde},  \eqref{eq:mult.one.h.pde} and again \cite[Lemma 9.6]{WangWei}, and absorbing the terms that are quadratic in $\phi$ we get:
\begin{multline} \label{eq:mult.one.phi.c2alpha.Oeps.intermediate}
	\Vert \phi \Vert_{C^{2,\alpha}_\eps(B_{19 - 4 \eps L}^\Gamma \times (-1,1))} + \Vert H_\Gamma - \Delta_\Gamma h \Vert_{C^{0,\alpha}_\eps(B^\Gamma_{19-4\eps L})} \\
	\leq \sigma (\Vert \phi \Vert_{C^{2,\alpha}_\eps(B_{19}^\Gamma \times (-1,1))} + \Vert H_\Gamma - \Delta_\Gamma h \Vert_{C^{0,\alpha}_\eps(B^\Gamma_{19})}) + C' \eps,
\end{multline}
for a fixed $C' > 0$. Iterating this procedure on $B_{19 - 4k \eps L}^\Gamma \times \mathbf{I}_\eps$ for $k = 1, \ldots, M |\log \eps|$, where $M$ depends on $\sigma \in (0, 1)$ but not $\eps$, yields the $\phi$ estimate in \eqref{eq:mult.one.h.phi.estimates} and thus also \eqref{eq:mult.one.mean.curv.estimates}.

We move on to verifying   \eqref{eq:mult.one.h.phi.horizontal.estimates}. Differentiating \eqref{eq:mult.one.phi.pde} in the directions parallel to $\Gamma$ (i.e., in $y_i$ in Fermi coordinates) we see similarly to \eqref{eq:mult.one.phi.pde.Oeps} that:
\begin{equation} \label{eq:mult.one.phi.horizontal.pde.Oeps2}
	\eps^2 \Delta (\eps \partial_{y_i} \phi) - W''(\overline \mbh_\eps) (\eps \partial_{y_i} \phi) = - \eps (\eps \partial_{y_i} (H_\Gamma - \Delta_\Gamma h)) \overline{\mbh}_\eps' + \cR \\
\end{equation}
where the error term can be estimated (using \eqref{eq:mult.one.h.phi.estimates}) by:
\[\Vert \cR \Vert_{C^{0,\alpha}_\eps} \leq C \eps^2 + C (\eps^2 \Vert \nabla^2_\Gamma \eps \partial_{y_i} \phi \Vert_{C^{0,\alpha}_\eps} + \eps \Vert \nabla_\Gamma \eps \partial_{y_i} \phi \Vert_{C^{0,\alpha}_\eps} + \Vert \eps \partial_{y_i} \phi \Vert_{C^{0,\alpha}_\eps})^2 
\]
Next, one differentiates \eqref{eq:mult.one.phi.orthogonal} in the horizontal directions to show, similarly as in \eqref{eq:mult.one.phi.first.alternative} but also estimating the error term $\langle \phi, \partial_{y_i} \overline \mbh_\eps' \rangle_{L^2}$, that
\begin{equation} \label{eq:mult.one.phi.horizontal.first.alternative}
	\left| \int_{-\eps K}^{\eps K} \eps (\partial_{y_i} \phi)(y, z) \overline \mbh_\eps'(y, z) \, dz \right| \leq C e^{-\sqrt{2} K} \Vert \eps (\partial_{y_i} \phi)(y, \cdot) \Vert_{C^0(\mathbf{I}_\eps)} + C \eps^3.
\end{equation}
Lemma \ref{lemm:Hprime-perp-estimate}'s first alternative can only hold for $\psi = \eps \partial_{y_i} \phi$, then, in case $\Vert \eps \partial_{y_i} \phi \Vert = O(\eps^3)$ (which is smaller than the worse upper bound we wish to prove, and thus does not break the applicability of our previous strategy).  Arguing as above, using \eqref{eq:mult.one.phi.horizontal.pde.Oeps2} 
instead of \eqref{eq:mult.one.phi.pde.Oeps} 
yields \eqref{eq:mult.one.h.phi.horizontal.estimates}.

Finally, we establish \eqref{eq:mult.one.phihat.estimate}. Recall that, by \eqref{eq:mult.one.mean.curv.estimates} and \eqref{eq:mult.one.phihat.pde}, $\hat \phi = \phi - \eps \mfh \mbi$ satisfies:
\begin{equation} \label{eq:mult.one.phihat.pde.Oeps2}
	\eps^2 \Delta \hat \phi - W''(\overline \mbh_\eps) \hat \phi = O_{C^{0,\alpha}_\eps}(\eps^2),
\end{equation}
The function $\hat \phi$ satisfies an estimate similar to \eqref{eq:mult.one.phi.horizontal.first.alternative}, namely:
\begin{equation} \label{eq:mult.one.phihat.first.alternative}
	\left| \int_{-\eps K}^{\eps K} \hat \phi(y, z) \overline \mbh_\eps'(y, z) \, dz \right| \leq C e^{-\sqrt{2} K} \Vert \hat \phi(y, \cdot) \Vert_{C^0(\mathbf{I}_\eps)} + C \eps^3.
\end{equation}
Thus, as before, Lemma \ref{lemm:Hprime-perp-estimate}'s first alternative can only hold for $\psi = \hat \phi$, then, in case $\Vert \hat \phi \Vert = O(\eps^3)$ (which is smaller than the worse upper bound we wish to prove, and thus does not break the applicability of our previous strategy). The rest of the argument goes through as before, applying \eqref{eq:mult.one.phihat.pde.Oeps2} and  \eqref{eq:mult.one.phihat.first.alternative} instead of \eqref{eq:mult.one.phi.horizontal.pde.Oeps2} and \eqref{eq:mult.one.phi.horizontal.first.alternative}.

\bibliographystyle{plain}
\bibliography{main}

\end{document}